\documentclass[11pt]{article}
\usepackage{sibarticle}
\usepackage{lmodern}
\usepackage{url}
\usepackage{nicefrac}
\usepackage{microtype}
\usepackage{amssymb,graphicx}
\usepackage{float,amsmath}
\usepackage{amsfonts,epsfig}
\usepackage[inoutnumbered,linesnumbered,algoruled,slide,vlined]{algorithm2e}
\usepackage{subcaption}
\usepackage{color}
\usepackage{xcolor}
\usepackage{hyperref}
\usepackage{cancel}
\usepackage[square,numbers]{natbib}

\renewcommand{\vec}[1]{{\mathbf{\boldsymbol{#1}}}}
\renewcommand{\d}{\mathrm{d}}
\newcommand{\R}{\mathbb{R}}
\newcommand{\E}{\mathbb{E}}

\DeclareMathOperator{\Prob}{Prob}

\DeclareMathOperator*{\argmin}{arg\,min}

\title{Generating Input Distributions for \\Explaining Portfolio Optimization Pipelines} 
\ShortTitle{Explaining Portfolio Optimization With Scenarios}
\ShortAuthors{Ataş, Aydın, Kıral, Birbil}

\NumberOfAuthors{4}

\FirstAuthor{Batuhan Ataş}
\FirstAuthorAddress{Faculty of Economics and Business,
University of Amsterdam, The Netherlands}

\SecondAuthor{Nurşen Aydın}
\SecondAuthorAddress{Warwick Business School,
University of Warwick, United Kingdom}

\ThirdAuthor{E. Mehmet Kıral}
\ThirdAuthorAddress{RIKEN AIP, Japan}

\FourthAuthor{Ş. İlker Birbil}
\FourthAuthorAddress{Faculty of Economics and Business,
University of Amsterdam, The Netherlands}


\allowdisplaybreaks

\begin{document}


\newcommand\hlb[3][]{ \todo[inline,caption={emptytext},
  size=\normalsize, backgroundcolor=yellow!70, bordercolor=yellow!70, noshadow, #1]{
    \begin{minipage}{
        \textwidth-4pt}#2
    \end{minipage}}
  \todo{\begin{spacing}{0.5}#3\end{spacing}}}

\newcommand{\hlc}[2]{\hl{#1}\todo{\begin{spacing}{0.5}#2\end{spacing}}}

\newcommand{\sitodo}[2][]{\todo[caption={#2}, #1]{
    \begin{spacing}{0.5}#2\end{spacing}}}

\newcommand{\intodo}[2][]{\todo[inline, noshadow, caption={emptytext}, #1]{{\bf \textcolor{blue}{TO-DO:}} \newline
    \begin{spacing}{1.0}\normalsize{#2}\end{spacing}}}

\newcommand{\sib}[1]{\textcolor{violet}{#1}}
\newcommand{\tsib}[1]{\begin{tcolorbox}\textcolor{violet}{#1}\end{tcolorbox}}

\newcounter{sibcmntcounter}
\setcounter{sibcmntcounter}{1}
\long\def\symbolfootnote[#1]#2{\begingroup
  \def\thefootnote{\fnsymbol{footnote}}\footnote[#1]{#2}\endgroup}
\newcommand{\sibcmnt}[1]{{\small\textbf{
      \textcolor{violet}{(C.\arabic{sibcmntcounter})}}
    \let\thefootnote\relax\footnotetext{\textcolor{violet}
        {\scriptsize(C.\arabic{sibcmntcounter})~#1}}}
  \addtocounter{sibcmntcounter}{1}}

\newcommand{\red}[1]{\textcolor{red}{#1}}
\newcommand{\blue}[1]{\textcolor{blue}{#1}}
\newcommand{\magenta}[1]{\textcolor{magenta}{#1}}
\newcommand{\rb}[1]{\raisebox{-1.5ex}[0cm][0cm]{#1}}
\newcommand{\HRule}{\noindent\rule{\linewidth}{0.5mm}}
\newcommand{\dsum}{\displaystyle\sum}
\newcommand{\veps}{\varepsilon}

\newcommand{\CA}{\mathcal{A}}
\newcommand{\CB}{\mathcal{B}}
\newcommand{\CC}{\mathcal{C}}
\newcommand{\CD}{\mathcal{D}}
\newcommand{\CG}{\mathcal{G}}
\newcommand{\CI}{\mathcal{I}}
\newcommand{\CJ}{\mathcal{J}}
\newcommand{\CK}{\mathcal{K}}
\newcommand{\CL}{\mathcal{L}}
\newcommand{\CN}{\mathcal{N}}
\newcommand{\CP}{\mathcal{P}}
\newcommand{\CS}{\mathcal{S}}
\newcommand{\CT}{\mathcal{T}}
\newcommand{\CU}{\mathcal{U}}
\newcommand{\CX}{\mathcal{X}}
\newcommand{\YY}{\mathbb{Y}}
\newcommand{\ZZ}{\mathbb{Z}}
\newcommand{\RR}{\mathbb{R}}
\newcommand{\NN}{\mathbb{N}}
\newcommand{\II}{\mathbb{1}}

\renewcommand{\vec}[1]{{\boldsymbol{\mathbf{#1}}}}

\newcommand{\va}{\vec{a}}
\newcommand{\vb}{\vec{b}}
\newcommand{\vc}{\vec{c}}
\newcommand{\vd}{\vec{d}}
\newcommand{\ve}{\vec{e}}
\newcommand{\vf}{\vec{f}}
\newcommand{\vg}{\vec{g}}
\newcommand{\vh}{\vec{h}}
\newcommand{\vp}{\vec{p}}
\newcommand{\vr}{\vec{r}}
\newcommand{\vt}{\vec{t}}
\newcommand{\vu}{\vec{u}}
\newcommand{\vw}{\vec{w}}
\newcommand{\vx}{\vec{x}}
\newcommand{\vy}{\vec{y}}
\newcommand{\vz}{\vec{z}}
\newcommand{\zv}{\vec{0}}
\newcommand{\ov}{\vec{1}}

\newcommand{\vveps}{\vec{\veps}}
\newcommand{\veta}{\vec{\eta}}
\newcommand{\vxi}{\vec{\xi}}
\newcommand{\valpha}{\vec{\alpha}}
\newcommand{\vbeta}{\vec{\beta}}
\newcommand{\vgamma}{\vec{\gamma}}
\newcommand{\vtheta}{\vec{\theta}}
\newcommand{\vlambda}{\vec{\lambda}}
\newcommand{\vnu}{\vec{\nu}}
\newcommand{\vpi}{\vec{\pi}}
\newcommand{\vtau}{\vec{\tau}}
\newcommand{\vSigma}{\vec{\Sigma}}
\newcommand{\vOmega}{\vec{\Omega}}
\newcommand{\vTheta}{\vec{\Theta}}
\newcommand{\vmu}{\vec{\mu}}

\newcommand{\mA}{\vec{A}}
\newcommand{\mB}{\vec{B}}
\newcommand{\mC}{\vec{C}}
\newcommand{\mD}{\vec{D}}
\newcommand{\mE}{\vec{E}}
\newcommand{\mF}{\vec{F}}
\newcommand{\mG}{\vec{G}}
\newcommand{\mH}{\vec{H}}
\newcommand{\mI}{\vec{I}}
\newcommand{\mL}{\vec{L}}
\newcommand{\mM}{\vec{M}}
\newcommand{\mP}{\vec{P}}
\newcommand{\mQ}{\vec{Q}}
\newcommand{\mR}{\vec{R}}
\newcommand{\mS}{\vec{S}}
\newcommand{\mU}{\vec{U}}
\newcommand{\mV}{\vec{V}}
\newcommand{\mX}{\vec{X}}

\newcommand{\tr}{^{\intercal}}
\newcommand{\ntr}{^{-\intercal}}
\newcommand{\inv}{^{-1}}

\newcommand{\gr}{\mbox{graph}}
\newcommand{\ra}{\rightarrow}
\newcommand{\la}{\leftarrow}
\newcommand{\Ra}{\Rightarrow}
\newcommand{\rra}{\rightrightarrows}
\newcommand{\ptr}{\marginpar{$\Leftarrow$}}

\newcommand{\pfxi}{\frac{\partial f(\vx)}{\partial x_i}}
\newcommand{\pfx}{\partial f(\vx)}
\newcommand{\pf}{\partial f}
\newcommand{\pxi}{\partial x_i}
\newcommand{\px}{\partial x}

\newcommand{\nfx}{\nabla f(\vx)}
\newcommand{\eps}{\epsilon}
\newcommand{\eg}{\textit{e.g.}}
\newcommand{\ie}{\textit{i.e.}}

\newcommand{\vsp}{\vspace{4mm}}
\newcommand{\vspp}{\vspace{8mm}}
\newcommand{\vsppp}{\vspace{12mm}}

\newcommand{\hsp}{\hspace{4mm}}
\newcommand{\hspp}{\hspace{8mm}}
\newcommand{\hsppp}{\hspace{12mm}}

\newcommand{\pr}[1]{\mathbb{P}\left(#1\right)}
\newcommand{\ex}[1]{\mathbb{E}\left[#1\right]}
\newcommand{\variance}[1]{\mbox{Var}\left(#1\right)}
\newcommand{\covar}[1]{\mbox{Cov}\left(#1\right)}
\newcommand{\C}[2]{\left(\begin{array}{c} #1 \\ #2 \end{array}\right)}

\newcommand{\maximize}{\mbox{maximize\hspace{4mm} }}
\newcommand{\minimize}{\mbox{minimize\hspace{4mm} }}
\newcommand{\subto}{\mbox{subject to\hspace{4mm}}}

\newenvironment{sibitemize}{
  \renewcommand{\labelitemi}{$\diamond$}
  \begin{itemize}
    \setlength{\parskip}{0mm}}
  {\end{itemize}}

\newcommand{\propnum}[2]{\vspace{3mm}
  \noindent {\sc Proposition #1}{\it #2} \vspace{3mm}}
\newcommand{\lemnum}[2]{\vspace{3mm}
  \noindent {\sc Lemma #1}{\it #2} \vspace{3mm}}
\newcommand{\thmnum}[2]{\vspace{3mm}
  \noindent {\sc Theorem #1}{\it #2} \vspace{3mm}}

\maketitle

\begin{abstract}
We propose a \emph{predict-optimize-explain} framework that uses gradient-based sample generation to interpret various portfolio models by identifying macroeconomic conditions that induce specified portfolio outcomes. Unlike traditional feature-importance methods, this approach directly probes decision pipelines (predictive models coupled with portfolio optimization) by constructing economically meaningful what-if questions. We focus on four such questions: under what macroeconomic conditions a predict-then-optimize pipeline closes or reverses its return gap with a predict-and-optimize pipeline; what conditions lead a pipeline to diversify rather than concentrate its allocation; when a pipeline trained on calm markets overtakes one trained through crises; and what conditions would let a pipeline match a benchmark return.
These examples illustrate how our framework uncovers key behavioral differences between various decision pipelines. Beyond these cases, the proposed framework is flexible and can support a wide range of probing questions tailored to specific portfolio objectives. Our findings highlight the value of integrating prediction, optimization, and explanation to produce more robust and transparent portfolio strategies.   
\end{abstract}

\section{Introduction.}
\label{sec:Intro}



Current portfolio optimization methods, such as mean–variance optimization (MVO), risk-based approaches, robust techniques, and the more recent learning-based methods, face ongoing challenges. These include difficulties in calibrating method parameters and, more critically, dealing with estimation uncertainty, especially in data-scarce environments \citep{Markowitz52, bertsimas2018data, Costa22, Blanchet22}. Traditionally, such uncertainty is handled via a sequential \textit{predict-then-optimize} (PTO) framework. A predictive model (often trained with neural networks) first estimates unknown quantities, which are then passed to an optimization model that prescribes the best action based on the forecasts \citep{Elmachtoub2022}. However, when predictions are noisy or when the decision problem is highly sensitive to input deviations, the two-step approach can lead to significant degradation in decision quality \citep{Donti17, Cong2021, Elmachtoub2022, Mandi24}. Recent research addresses this misalignment through \textit{decision-focused learning} (DFL), also known as the \textit{predict-and-optimize} (PAO), where the predictive model is trained with the downstream optimization objective in mind. Rather than optimizing predictive models solely for statistical loss functions (e.g., mean squared error), DFL incorporates the decision objective directly into the training process  \citep{Elmachtoub2022, Donti17}. By backpropagating portfolio performance through the prediction stage, it encourages the model to ``learn to optimize,'' emphasizing signals that are most critical for the final allocation decision. While this improves alignment between prediction and optimization, it also makes portfolio pipelines harder to interpret: changes in the training objective, optimization layer, or uncertainty treatment may induce sharp and state-dependent changes in portfolio decisions.

\noindent\textbf{Motivation.}  We view modern portfolio construction as a \emph{predict–optimize–explain} (POE) pipeline. Small changes in training objectives (e.g., forecast error vs. portfolio-level targets), optimization layers (classical vs. robust), or uncertainty modeling can induce materially different allocations. Existing evaluations, based mainly on aggregate out-of-sample performance or average-case bounds, provide limited insight into \emph{when} and \emph{why} different pipelines diverge \citep{Lee25, Uysal24}. This gap is most consequential in portfolios, where macroeconomic states and small input perturbations can trigger sharp allocation shifts. Explaining such pipelines therefore requires moving beyond average performance comparisons toward a structured analysis of the input regions that induce particular decision-level behaviours.

\noindent\textbf{Our approach and contributions.}
We propose to explain decisions by generating \emph{distributions over inputs} via our POE framework. Instead of asking how a given macroeconomic state maps to portfolio weights, we invert the question: under which macroeconomic scenarios does a specified decision-level event occur (e.g., ``When does a stress-oriented policy outperform a calm-oriented one?'' or ``When does PTO close the gap to PAO?''). We cast this as a variational probing problem that guides a prior, estimated from the macroeconomic process, toward input regions that induce the target portfolio behavior. The resulting distributions provide economically plausible ``explanatory'' scenarios rather than arbitrary perturbations. This approach paves the way for our contributions:
\begin{sibitemize}
    \item \textbf{Explanation as input distributions.} We introduce a general, model-
agnostic variational framework for explaining portfolio decision pipelines through distributions over macroeconomic inputs. Rather than attributing a single portfolio decision to local perturbations around one observed state, the proposed framework identifies regions of the macroeconomic environment under which trained pipelines exhibit specified decision-level outcomes. A data-driven prior keeps the generated states close to realistic market conditions.
    \item \textbf{Theory: uniqueness and stability.} We provide a theoretical characterization of the proposed explanatory framework. For a fixed probing question and temperature parameter, the explanatory distribution is unique and admits a closed-form characterization (Proposition~\ref{prop:existence}). We also show that the framework has a Bayesian interpretation. Moreover, the resulting explanations are stable: any well-behaved probing function encoding the same economic event concentrates on the same scenario set, up to smoothing (Proposition~\ref{prop:robustness}, Corollary~\ref{cor:unique}). This ensures that differences in explanations reflect differences in the underlying economic questions being asked.
    \item \textbf{Empirical case study.} We apply the framework to an open-source cross-sectional firm--month asset-pricing panel. We compare PTO and PAO portfolio pipelines that share the same predictive network and optimization layer, differing only in their training. This controlled comparison isolates how decision-focused training changes portfolio behavior relative to a sequential prediction-based pipeline. We release our entire \href{https://github.com/sibirbil/Predict-Optimize-Explain}{code base} for reproducibility. 
    \item \textbf{Interpretability in practice.} We show how the generated explanatory distributions can be translated into interpretable economic insights. Specifically, we map the generated scenarios to economic regimes using a classifier and link them to historical analogues through nearest months in the National Financial Conditions Index (NFCI). This grounds the explanatory distributions to interpretable market regimes and historical episodes, making the results more meaningful for portfolio analysis.
\end{sibitemize}

\section{Related Literature.}
\label{sec:LitRew}
In recent years, advances in machine learning and data availability have spurred a wave of learning-based portfolio optimization methods that integrate predictive models with decision-making. Such methods are used not only to improve forecasts of asset returns \citep{GuKellyXiu2020}, but also to guide portfolio allocation decisions through PTO and PAO training strategies that integrate learning models with a Markowitz portfolio optimization layer. In PTO training strategy, estimation errors in the prediction stage are blind to their downstream consequences in the optimization stage. Indeed, it has been formally noted that good prediction performance does not always guarantee good decision performance, and formal guarantees linking the two have been lacking. \citeauthor{Nguyen22} derive conditions under which improving prediction quality translates to improved optimization outcomes \citep{Nguyen22}. Notably, they show that if a prediction model satisfies certain consistency conditions, the optimality gap of decisions in terms of prediction error can be bounded. 

End-to-end decision-focused learning, on the other hand, seeks to overcome the PTO gap by integrating the portfolio optimization step into the training of the predictive model. \citeauthor{Elmachtoub2022} introduce the seminal ``Smart Predict-then-Optimize” framework, defining a loss equal to the cost of implementing decisions based on predictions versus perfect information \citep{Elmachtoub2022}. They show that minimizing this decision error loss can lead to better out-of-sample decisions, even if it sacrifices some predictive accuracy. Following this insight, a growing body of work develops methods that differentiate through optimization problems to enable end-to-end training of portfolio models. For example, \citeauthor{Butler21} apply decision-focused learning to classical portfolio optimization and show that directly training models on the portfolio objective improved performance over decoupled predictive modeling \citep{Butler21}. They derive analytical solutions for regression models integrated in mean–variance optimization, effectively solving for parameters that minimize the realized portfolio cost. \citeauthor{Uysal24} similarly compare two end-to-end approaches for risk budgeting: a model-based network that embeds an optimization layer to allocate assets based on learned risk contributions, and a model-free network that directly outputs portfolio weights from raw features \citep{Uysal24}. Both end-to-end models outperformed a traditional two-step approach, with the model-based variant achieving higher risk-adjusted returns. \citeauthor{Anis25} extend the end-to-end paradigm to more complex decision-constrained problems, specifically a portfolio with a cardinality constraint. They embed the entire mixed-integer optimization problem into the learning process by differentiating through continuous relaxations of the discrete constraint. Their end-to-end learned factor model, trained on a portfolio volatility objective, delivered significantly lower out-of-sample portfolio risk than traditional factor estimation across various sparsity levels.

Beyond optimization-layer design, several studies emphasize the role of model architecture in aligning learning with portfolio-level objectives. \citeauthor{Zhang21} propose a universal end-to-end framework to optimize portfolio-level objectives (e.g., maximizing Sharpe ratio or a mean-variance utility) within the model’s training loss \citep{Zhang21}. They integrate Long Short-Term Memory (LSTM) networks to model temporal price patterns, ensuring the model recognizes trends and mean-reversion in asset returns. Along similar lines, \citeauthor{KisielGorse2022} introduce the Portfolio Transformer, an attention-based neural architecture trained end-to-end to directly maximize risk-adjusted returns rather than predict prices \citep{KisielGorse2022}. By formulating the Sharpe ratio as the training objective, their model learns to allocate assets in one unified step. These approaches highlight the growing role of architectural design in aligning learning with portfolio-level objectives.  Robustness considerations have also been incorporated directly into learning-based portfolio systems. \citeauthor{Costa22} propose a distributionally robust end-to-end portfolio system in which a differentiable prediction layer feeds into a distributionally-robust optimization layer, with both risk-aversion and uncertainty parameters learned from data \citep{Costa22}. By modeling an ambiguity set around the return distribution, their framework regularizes the decision and improves out-of-sample performance. \citeauthor{Wang25} introduce Gen-DFL, which learns a distribution over uncertain parameters and samples worst-case tail events during training \citep{Wang25}. This generative decision-focused approach provably improves worst-case performance bounds and yields more robust decisions under high uncertainty. \citeauthor{KimDSL2025} propose a practical portfolio optimization framework that reframes portfolio construction as a supervised learning problem \citep{KimDSL2025}. Their approach generates decision-optimal portfolio weights by solving Sharpe-ratio–based optimization problems offline and then trains deep neural networks to map market features directly to these allocations. By combining multiple networks through deep ensembles, their aim is to improve the stability and robustness of portfolio weights. Practical applications have also begun to leverage advanced architectures. \citeauthor{Hwang25} integrate large-language-model embeddings and an attention mechanism with a portfolio optimization layer, showing that minimizing traditional loss alone is suboptimal and that decision-focused training produces better portfolios \citep{Hwang25}.

Despite these advances, the decision-focused learning in financial decision-making share a common limitation, the resulting models often function as black boxes. The economic rationale behind portfolio allocations, whether driven by inflation expectations, volatility regimes, valuation ratios, or other macroeconomic forces, typically remains opaque. In portfolio management, practitioners and regulators typically demand explanations for why a model recommends a particular allocation especially during volatile periods. \citeauthor{Avramov2023} and \citeauthor{Israel2020} argue that without economic interpretability, machine learning models introduce significant model risk \citep{Avramov2023, Israel2020}. Existing interpretability methods in the relevant literature primarily operate at the prediction level. A prominent line of work focuses on model distillation, where complex black-box predictors are approximated by simpler and more transparent models, such as decision trees or rule lists \citep{Bastani2019}. While such approaches improve transparency by revealing how inputs map to predictions, they do not directly explain the behavior of downstream optimization decisions. In asset pricing, \citeauthor{GuKellyXiu2020} use variable importance measures to rank firm characteristics and identify leading drivers of cross-sectional return variation \citep{GuKellyXiu2020}. However, identifying which variables are influential on average is fundamentally different from understanding how a trained model responds to different scenarios. Related work has begun to connect interpretability with decision-making. \citeauthor{Elmachtoub20} develop decision trees trained under decision-aware loss functions, yielding interpretable feature partitions that correspond to distinct optimal decisions \citep{Elmachtoub20}. While this provides transparency in how features map to decisions, the analysis remains tied to explicit tree structures. Similarly, \citeauthor{Chen24} visualize nonlinear interactions among firm characteristics learned by deep models, but their focus remains on return prediction rather than portfolio allocation or decision behavior \citep{Chen24}. A closely related line interprets trained portfolio policies through economic distillation: \citeauthor{Cong2021} project a deep-reinforcement-learning portfolio policy onto interpretable firm-level characteristics to recover the dominant drivers of its allocations \citep{Cong2021}. Such methods explain a policy in terms of which firm characteristics it rewards. Our framework is complementary: it explains a fixed decision pipeline in the macroeconomic input space, is agnostic to the predictive architecture and optimization layer, and, because it operates at the decision level, can compare competing pipelines by the macroeconomic conditions under which their decisions coincide or diverge.

Several recent studies extend attribution-based interpretability to portfolio problem. \citeauthor{Franco25} combine a portfolio construction technique with explainable artificial intelligence (XAI) methods to identify which assets drive portfolio returns over time \citep{Franco25}. Similarly, \citeauthor{Agal25} apply XAI method to deep portfolio models to attribute the importance of various input features (e.g., technical indicators, macro signals) on the resulting allocation \citep{Agal25}. These XAI-based approaches offer useful post-hoc diagnostics but remain centered on feature attribution rather than decision-level behavior. Our contribution diverges from both attribution-based and structure-based interpretability methods. Rather than asking which inputs are important on average or approximating the model with an interpretable surrogate, we adopt a POE perspective that centers explainability at the decision level. We treat interpretation as an inverse problem on the input space. By solving a variational problem to generate distributions over inputs that induce specific portfolio behaviors, we provide a generic and decision-level tool for characterizing the economic logic of a trained portfolio model. This approach enables systematic analysis of learned allocation rules, distinguishing our framework from purely statistical or attribution-based explainability techniques.

\section{Methodology.}
\label{sec:Method}
This section formalizes our probing framework used to analyze decision pipelines for portfolio construction. The framework can be considered in two stages: training process and generation process.

We start with the first stage as illustrated in Figure \ref{fig:forwardCalc}\footnote{Boxes with dashed outlines indicate that the values inside them are optional and may or may not be used. Similarly, a dashed arrow represents an optional connection or inclusion.}. For a given time horizon, $t=1, \dots, T$ and a selection of assets, $a=1, \dots, A$, we combine the firm characteristics, $\vec{c}_{a, t}$ and the macroeconomic variables $\vec{m}_t$ into feature vectors $\vec{x}_{a,t}$ with interaction terms, see e.g., \citep{GuKellyXiu2020} where this exact input preparation is employed. A neural network with trainable parameters, $\vec{\theta}$ takes these features as input and predicts a vector of excess returns, $\widehat{\vec{r}}_{t+1}$ for the next period. This vector is then passed to a portfolio optimization layer which solves robust optimization models of the form:
\begin{equation}
\label{eqn:genopt}
\begin{array}{ll}
\maximize & \vec{w}^\top \widehat{\vec{r}}_{t+1} -\kappa \sqrt{\vec{w}^\top\Omega \vec{w}} - \tfrac{\lambda}{2}\vw^\top \vSigma \vw  \\
\subto & \vw^\top\ve = 1, \\
& \vw \geq \zv,
\end{array}
\end{equation}
where $\vSigma$ is the covariance matrix of returns, $\ve$ is the vector of ones, $\vOmega$ is a symmetric positive definite matrix defining the uncertainty set\footnote{We consider the quadratic uncertainty set in the robust portfolio optimization model. This choice is also advocated by \citeauthor{yin2021practical} who have based their reasoning to the existing literature and their experience in practice \citep{yin2021practical}.}, $\lambda \geq 0$ is the risk aversion parameter, and $\kappa \geq 0$ controls the level of uncertainty. Note that $\kappa=0$ corresponds to the celebrated MVO problem. This formulation intentionally follows \citeauthor{yin2021practical}, where the authors also discuss various practical choices of uncertainty matrices $\vOmega$, each leading to a different robust portfolio optimization problem deployed in implementations of a large European bank \citep{yin2021practical}.

\begin{figure}[http]
    \centering
    \includegraphics[width = \linewidth]{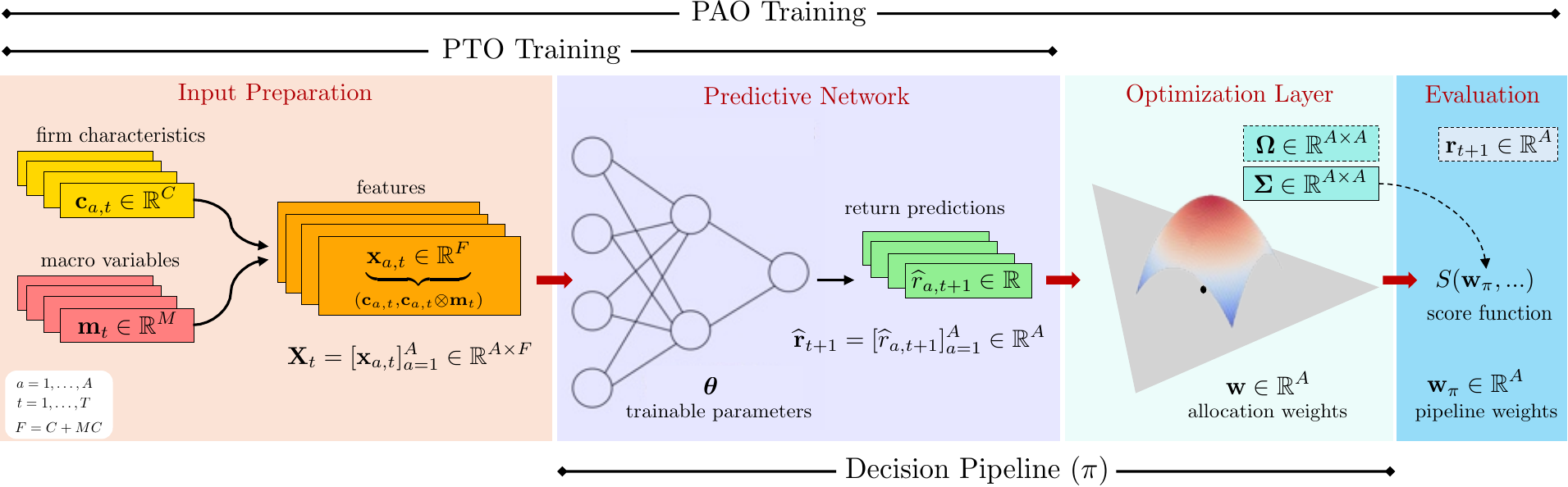}
    \caption{Training Process}
    \label{fig:training_process}
    \label{fig:forwardCalc}
\end{figure}

After obtaining the asset allocation weights from the optimization layer, portfolio-level quantities such as total return, risk-adjusted return, or the Sharpe ratio can be calculated via score function, $S$. Note that Figure \ref{fig:forwardCalc} also shows the extents of training both PTO and PAO strategies. Both strategies use the same input data and the predictive network, but differ in how the network parameters are learned. PTO trains the network using loss functions that target accurate return prediction and \emph{then} feeds the resulting predictions into the optimization layer, whereas PAO uses the optimization layer \emph{and} the score function to train the network parameters. In the subsequent sections of our paper, a decision pipeline (designated by $\pi$ in the figure) refers to this composite mapping from macroeconomic conditions and firm characteristics to portfolio allocations, independent of whether the parameters were obtained via a PTO or a PAO training strategy. The goal of our methodology is to probe such a decision pipeline at the level of its economic inputs and portfolio outputs for a given time.

This takes us to the second stage, i.e., the generation process. In Figure \ref{fig:backwardCalc}\footnote{Boxes with dashed outlines indicate that the values inside them are optional and may or may not be used. Similarly, a dashed arrow represents an optional connection or inclusion.}, we show this process in two parts marked as (a) and (b) for a given time point $\bar{t}$. At the end of stage one, we have obtained a \emph{trained} neural network. This trained network is marked in part (a) with the fixed parameter, $\vec{\theta}^*$ and it induces the decision pipeline $\pi(\vtheta^*)$. To condense the notation, we simply use $\pi$ in place of $\pi(\vtheta^*)$, note that $\pi$ also depends on the particular opptimization layer (choice of $\kappa$ and $\lambda$ in \eqref{eqn:genopt}) which we also suppress. Our framework explores the behavior of the decision pipeline, $\pi$ by generating macroeconomic scenarios, $\vec{m}$, specifically designed to shed light on targeted portfolio-related questions. At the core of this process is a user-defined probing function, $G$ which is defined on macroeconomic variables and induces a landscape over the space of macro conditions. By translating this landscape into a Gibbs distribution, our framework defines a probability distribution, $p_\pi$ that assigns higher likelihood to macro states where the probing function is low, while still preserving sufficient dispersion to explore the input space. This is illustrated in part (b) of Figure \ref{fig:backwardCalc}. Sampling from this distribution then yields macroeconomic scenarios that allow us to address questions such as: What macroeconomic conditions, given a specific cross-section of firms, lead to predicted portfolio returns or risk-adjusted returns that closely match a benchmark? Under which macroeconomic scenarios does the model produce diversified versus concentrated asset allocations? In which economic regimes do two different allocation strategies—such as a predict–then–optimize versus a predict–and–optimize approach—diverge the most in their resulting portfolio weights? These three questions are instantiated more formally in Section \ref{subsec:probingfunctions}.

\begin{figure}
    \centering
    \includegraphics[width=\linewidth]{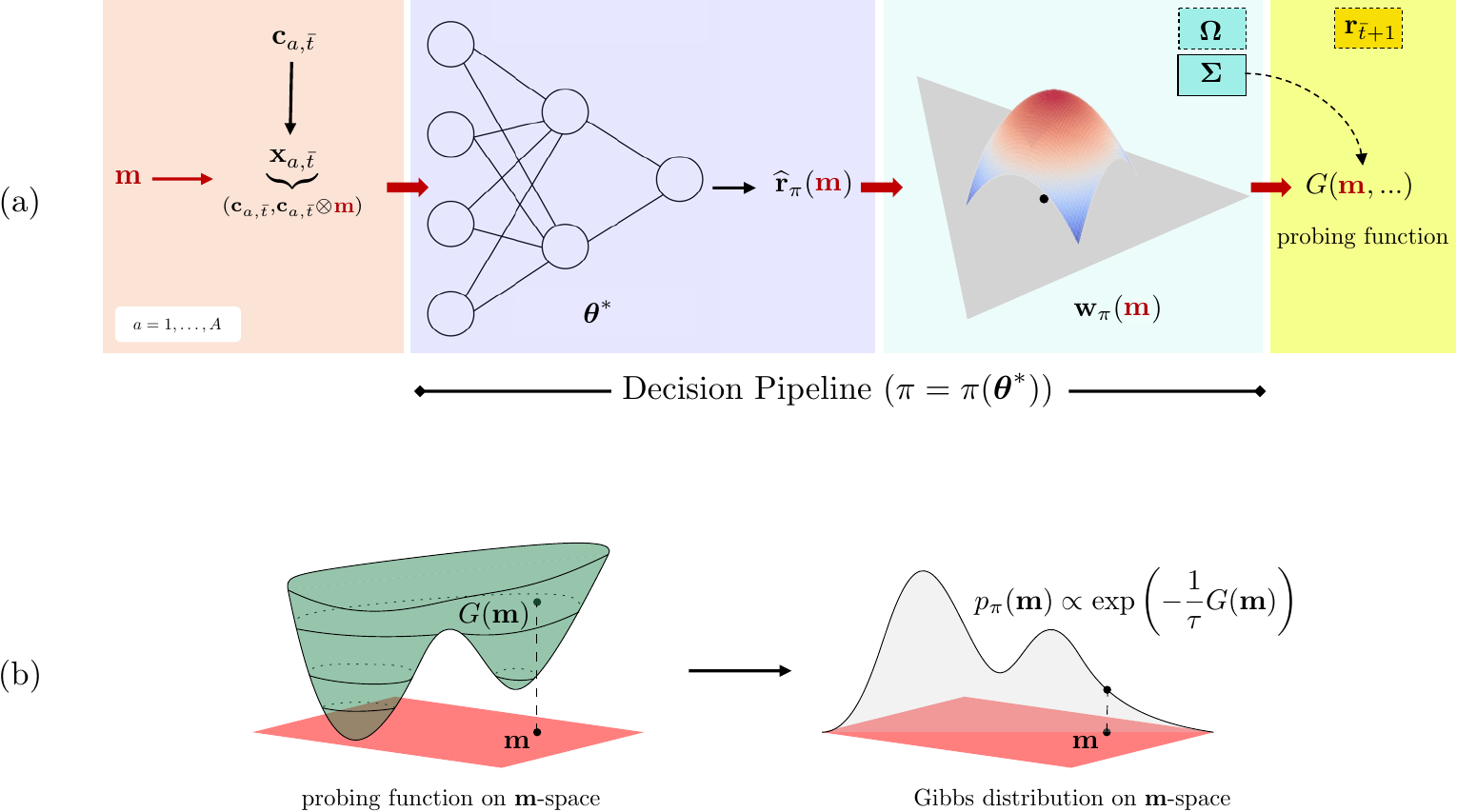}
    \caption{Generation Process}
    \label{fig:backwardCalc}
\end{figure}

The following subsections provide the details of this probing framework. We describe the underlying learning–optimization models, develop the variational and sampling machinery, and discuss the concrete portfolio setups and probing questions used in our empirical analysis. 

\subsection{Training Process.}
Machine learning comprises a broad class of methods designed to detect patterns in large-scale datasets. As a canonical example, a standard feed-forward neural network defines a parametrized family of functions $f_{\vec{\theta}}$. 
\emph{Training} such a network on a dataset $\{(\vec{X}_t, \vec{r}_{t+1})\}_{t = 1}^T$ means finding a parameter vector $\vec{\theta}^*$ such that the function becomes an approximate predictor $f_{\vec{\theta}^*}(\vec{X}_t) \approx \vec{r}_{t+1}$ on the training data, with the aim of generalizing to previously unseen observations.
The optimal parameter vector $\vec{\theta}^*$ is typically obtained via gradient-based optimization methods 
applied to the so-called \emph{empirical risk} function
\begin{equation}
\label{eq:EmpiricalRiskMinimizationF}
    \vec{\theta}^* = \argmin_{\vec{\theta}} \frac{1}{T} \sum_{t = 1}^T \ell(f_{\vec{\theta}}(\vec{X}_t), \vec{r}_{t+1}) + R(\vec{\theta}),
\end{equation}
where $\ell$ denotes a loss function that measures the discrepancy between the predicted and true labels, and $R$ represents a regularization term that effectively restricts the search space for the optimal parameter $\vec{\theta}^*$. 

When the PTO training strategy is deployed (see Figure \ref{fig:forwardCalc}) $\ell$ is taken as the mean square loss so that $f_{\vec{\theta}^*}(\vec{X}_t) \approx \vec{r}_{t + 1}$. The decision pipeline takes $\vec{X}_t$ as firm-level features at a given time $t$ and outputs its estimate of the returns for the next time. This is then followed by an optimization layer that performs portfolio asset allocation based on the predicted returns and the firms' covariances. The optimization layer can be chosen to be either a standard mean-variance optimization or robust optimization model, it can include long-only constraints or allow for both long and short positions. The input features typically consist of numerical vectors of firm characteristics, past returns over a fixed time window, or a combination thereof. We take it as the vector of interaction terms of firm characteristics and macroeconomic variables at time $t$ as in Figure \ref{fig:forwardCalc}.

The PAO training  strategy uses the same firm-level input features and predictive architecture. However, rather than training the model to minimize prediction error directly, the learning objective is defined at the portfolio level. In this setting, the model parameters $\vec{\theta}^*$ are optimized to minimize a loss function that uses a score function $S$ on the outputs $\vec{w}_\pi$ of the decision pipeline. We maximize the following score functions--or minimize their negatives: The (risk-adjusted) realized portfolio return is
\begin{equation}\label{eq:Score_RiskAdjRet}
    S(\vec{w}_\pi) = \vec{w}_\pi^\top \vec{r}_{t + 1} - \tfrac{\lambda}{2} \vec{w}^\top \Sigma \vec{w}.
\end{equation}
If $\lambda = 0$, this is simply the realized excess return of the porftolio decided by $\pi$ using the previous month's data. Alternatively, we train $\vec{\theta}^*$ so that the decision pipeline maximizes the Sharpe ratio of the portfolio,
\begin{equation}\label{eq:Score_Sharpe}
    S(\vec{w}_\pi) = \frac{\vec{w}_\pi^\top \vec{r}_{t + 1}}{\sqrt{\vec{w}_\pi^\top \Sigma \vec{w}_\pi}}.
\end{equation}
The return predictions influence the objective only indirectly, and prediction accuracy is not explicitly optimized.

\subsection{Generation Process.}
\label{subsec:generation}

We start with a decision pipeline, $\pi$ that has been obtained using one of the approaches described in training process and seek to characterize its behavior by identifying distributions over input variables $\vec{m} \in \mathcal{M}$. In our example $\mathcal{M} \subseteq \R^M$, c.f. Figure \ref{fig:forwardCalc}, we choose to restrict our input space to a box of the historical range of these variables. To this end, we define a \emph{probing function} $G$ on the data space. Rather than searching for a local minimum,
we characterize a \emph{distribution} from a loss function via the variational formulation of \eqref{eq:EmpiricalRiskMinimizationF}. 

Specifically, given a user-defined probing function $G: \mathcal{M} \to\R$ and a prior distribution $p_0$, we solve for a distribution on the input space $\mathcal{M}$ that optimizes
\begin{equation}
    \label{eq:BLPforG}
    p_{\pi, \tau} = \argmin_p \E_{p}[G] + \tau \mathrm{KL}(p\|p_0),
\end{equation}
where $\E_p[G] = \int G(\vec{m})p(\vec{m}) \d\vec{m}$ is the expected loss, and the proximal term $\mathrm{KL}(p\|p_0)$ is the Kullback-Leibler divergence between the distributions $p$ and the prior $p_0$, defined by $\operatorname{KL}(p\|p_0) = \int_{\mathcal{M}} p(\vec{m}) \log (\tfrac{p(\vec{m})}{p_0(\vec{m})}) \d \vec{m}$. We can express this in another way, as minimizing $\E_p[G - \tau\log p_0] - \tau\mathcal{H}(p)$ where $\mathcal{H}(p) = -\int p(\vec{m}) \log(p(\vec{m}))\d\vec{m}$ is the entropy. Note that minimizing \eqref{eq:BLPforG} yields a distribution over the parameter space rather than a single point estimate. In this formulation, after picking the function $\widetilde{G}(\vec{m}) := G(\vec{m}) -\tau\log p_0(\vec{m})$ and solving \eqref{eq:BLPforG}, we obtain a probability distribution $p_\pi$ that assigns higher mass on regions where $\widetilde{G}$ attains low values, while maintaining wider distribution over the parameter space. The temperature constant $\tau >0$ controls the trade-off between these competing objectives. The expected loss term is linear in $p$, while the negative entropy term is strictly convex, implying that the problem actually has a unique solution over the space of all distributions, given by the Gibbs-Boltzmann distribution
\begin{equation}\label{eq:GibbsforG}
    p_{\pi, \tau}(\vec{m}) = \frac{1}{Z} e^{-\frac{1}{\tau}\widetilde{G}(\vec{m})} = \frac{1}{Z} e^{-\frac{1}{\tau} G(\vec{m}) + \log p_0(\vec{m})},
\end{equation}
where $Z = \int e^{-\frac{1}{\tau}\widetilde{G}(\vec{m})} \d\vec{m}$ is the normalizing factor, as long as it is finite. Notice that for temperatures $\tau \in (0,\infty)$, the distribution $p_\pi$ has higher likelihood in regions where the probing function $G$ has low values. This construction yields a distribution over $\mathcal{M}$ that balances minimizing the expected value of $G$, which we are flexible to choose, and the entropy of the distribution. The resulting distribution serves as the basis for sampling input points $\vec{m}$ that reveal targeted behaviors of the trained model. The negative log-likelihood of the prior takes the role of the regularizer $R(\vec{m}):= -\log p_0(\vec{m})$. We generate scenarios that satisfy our desired criteria of being proximal to $p_0$ and also those guided by $G$, by sampling from $p_\pi$. 

\paragraph{Plausibility prior.}
The prior $p_0$ carries the economic plausibility of the generated scenarios. We construct it from the macroeconomic process itself: we fit a first-order vector autoregression to the historical macroeconomic series,
\[
\vec{m}_{t+1} = \Phi\,\vec{m}_t + \vec{c} + \vec{\varepsilon}_t,
\qquad \vec{\varepsilon}_t \sim \mathcal{N}(\vec{0}, \Sigma_\varepsilon),
\]
and take $p_0$ to be the stationary law of this process, $p_0 = \mathcal{N}(\vec{\mu}^*, \Sigma^*)$, with mean $\vec{\mu}^* = (I-\Phi)^{-1}\vec{c}$ and covariance $\Sigma^*$ solving the discrete Lyapunov equation $\Sigma^* = \Phi\,\Sigma^*\Phi^\top + \Sigma_\varepsilon$, restricted and renormalized on the macro box $\mathcal{M}$ so that the assumptions of Proposition~\ref{prop:existence} hold. This construction is anchor-independent: it encodes the long-run plausibility of macroeconomic configurations under the estimated dynamics rather than proximity to any particular date. When an explanation is sought relative to a specific anchor month $\bar{t}$, we localize the prior by multiplying the long-run plausibility of a neighbor state by a penalty for its distance from the anchor month:
\[
p_0(\vec{m}) \;\propto\;
\mathcal{N}\!\big(\vec{m};\, \vec{\mu}^*, \Sigma^*\big)\,
\exp\!\Big(-\tfrac{\rho}{2}\,(\vec{m}-\vec{m}_{\bar{t}})^\top (\Sigma^*)^{-1} (\vec{m}-\vec{m}_{\bar{t}})\Big).
\]
The product of these two Gaussian factors is itself Gaussian, with mean $(\vec{\mu}^* + \rho\,\vec{m}_{\bar{t}})/(1+\rho)$ and covariance $\Sigma^*/(1+\rho)$. The localized prior therefore centers on a weighted average of the long-run macroeconomic mean and the anchor state. The weight $\rho \geq 0$ acts as a dial between the two: $\rho = 0$ recovers the global stationary prior, larger values pull the prior toward the anchor month, and in the limit the prior concentrates at the anchor itself. We use $\rho = 1$ in our experiments, which places the prior mean halfway between the long-run mean and the anchor and halves the prior covariance. An important consequence of this construction is that proximity to the anchor is enforced by the prior alone: the probing functions $G$ of Section~\ref{subsec:probingfunctions} need no additional distance penalty and remain pure decision-level losses, consistent with the variational form~\eqref{eq:BLPforG}.

There are several methods for sampling from the Gibbs distribution, such as Metropolis Hastings Random Walk \citep{hastings1970monte}, Hamiltonian Monte Carlo \citep{neal2011mcmc}, or Picard Iterations to implement Langevin diffusion \citep{anari2024fast}. We focus on a particular Markov Chain Monte Carlo (MCMC) method that generates samples from the target distribution called the Metropolis Adjusted Langevin Algorithm (MALA). It is an MCMC method designed to sample from a target density of the form \eqref{eq:GibbsforG}. The theory of Markov chains and their limiting distributions is a well studied area of mathematics, see \citep{meyn2012markov}. The practitioner can ensure that the chain converges to a target distribution by desigining the chain's transitions. MALA uses discretized Langevin dynamics to propose new points, and the ``Metropolis Adjusted'' part of the algorithm means that these proposals are accepted or rejected based on the ratio of the target density between the current point and the proposed new point. Specifically, given the current point $\vec{m}^{(k)}$ at iteration $k$ and a step size $\eta > 0$, MALA proposes a new point according to
\begin{equation}\label{eq:MALAproposal}
    \vec{m}_{\text{prop}} = \vec{m}^{(k)} - \eta \nabla \widetilde{G}(\vec{m}^{(k)}) + \sqrt{2\tau \eta} \vec{\xi}^{(k)},
\end{equation}
where $\vec{\xi}^{(k)}\sim \mathcal{N}(\vec{0}, I)$ is a standard Gaussian noise vector and $\tau$ is the temperature coefficient. Notice that the higher the temperature, the more noise that is injected at each proposal. In the absence of the noise term, the dynamics reduce to standard gradient descent.
Each proposal is followed by a Metropolis-Hastings acceptance step. The proposed point is accepted (\textit{i.e.}, $\vec{m}^{(k+1)} = \vec{m}_\text{prop}$) or rejected (\textit{i.e.}, $\vec{m}^{(k+1)} = \vec{m}^{(k)}$) with probability
\begin{equation}\label{eq:ARratio}
    \alpha = \min \left\{1, \frac{p_{\pi, \tau}(\vec{m}_{\text{prop}}) \Gamma(\vec{m}_{\text{prop}}\to \vec{m}^{(k)} ) }{p_{\pi, \tau} (\vec{m}^{(k)})\Gamma(\vec{m}^{(k)} \to \vec{m}_{\text{prop}})}\right\},
\end{equation}
where $\Gamma(\vec{m} \to \vec{m}')$ denotes the transition probabilities of our proposals going from $\vec{m}$ to $\vec{m}'$. This step enforces the detailed balance condition, ensuring that $p_\pi$ is an invariant distribution of the Markov chain; see the classical work of \citeauthor{roberts1996exponential} for the convergence speed of this process \citep{roberts1996exponential}. The exact acceptance-rejection ratio given by \eqref{eq:ARratio} ensures the detailed balance condition, i.e. that the target distribution \eqref{eq:GibbsforG} is invariant under the updates.

If the proposal distribution were symmetric (\textit{i.e.}, $\Gamma(\vec{m} \to \vec{m}') = \Gamma(\vec{m}' \to \vec{m})$), the transition densities would cancel, and we would only compare densities of the current point and the proposal. In practice, it is sufficient to work with unnormalized densities, such as $e^{-\frac{1}{\tau}G(\vec{m})}p_0(\vec{m})$, since normalizing constants cancel in the acceptance ratio. For the proposal mechanism defined in \eqref{eq:MALAproposal}, the transition probabilities are given by
\[
	\Gamma(\vec{m} \to \vec{m}') = e^{-\frac{1}{\sqrt{8\tau \eta}}\left\|\vec{m}' - \vec{m} + \eta \nabla \widetilde{G}(\vec{m}))\right\|^2}.
\]
Since the target distribution $p_\pi$ and the proposal density are expressed in exponential form, numerical stability can be improved by computing
\[
	A = \frac{1}{\tau} ( \widetilde{G}(\vec{m}^{(k)}) - \widetilde{G}(\vec{m}_{\text{prop}}) ), 
\]
and 
\begin{equation}
B = \frac{1}{\sqrt{8\tau\eta}}
\Bigl(
\lVert \vec{m}_{\text{prop}}-\vec{m}^{(k)}+\eta\nabla \widetilde{G}(\vec{m}^{(k)})\rVert^2 
-\lVert \vec{m}^{(k)} - \vec{m}_{\text{prop}}+\eta\nabla \widetilde{G}(\vec{x}_{\text{prop}})\rVert^2
\Bigr)
\end{equation}
after we construct a proposal $\vec{x}_{\text{prop}}$.
Then, the acceptance probability can be evaluated as $\log\alpha = \min\left\{0, A + B\right\}$.
If a random uniform variable drawn from $[0,1]$ is less than $\alpha = \exp (\log\alpha)$, we accept the proposal, otherwise, we reject it.

Samples generated by an MCMC process are not independent, as the method proposes new points in the vicinity of the current point. To reduce this so-called autocorrelation, common strategies include running multiple independent MALA chains and applying thinning to retain approximately uncorrelated samples. We include some references and details on the convergence of the MALA to the target distribution in Appendix \ref{app:convergenceMALA}.

\subsection{Probing Functions.}
\label{subsec:probingfunctions}
Having obtained a trained model with optimal parameters, we seek to analyze and interpret its behavior using a probing function $G$, defined on the input space such that it attains low values whenever a prescribed property of the model is satisfied. Sampling from the corresponding Gibbs–Boltzmann distribution then allows us to explore the distribution of inputs that satisfy the desired conditions.

To illustrate, consider a setting where we compare two decision pipelines. These pipelines may differ in the optimization layer (e.g., mean–variance versus robust optimization), in the training objective (e.g., maximizing return versus Sharpe ratio), or in the datasets they were trained with.
We refer to the resulting decision pipelines as $\pi_1$ and $\pi_2$, and use $\pi$ to denote either one generically. Recall that  $\vec{w}_{\pi}(\vec{m})$ denotes the portfolio weights obtained under macroeconomic conditions $\vec{m}$, and $\vec{r}_{\bar{t} + 1}$ is the realized returns at time $\bar{t}+1$. Within this framework, we can pose a range of probing reverse questions.

\vspace{2mm}

\paragraph{Targeting a benchmark return.} Portfolio managers often work with explicit performance targets, such as benchmark excess returns, or minimum Sharpe levels. Identifying which macroeconomic conditions lead a decision pipeline to meet a given target supports mandate design and stress testing. By inverting the pipeline and searching for states that achieve a specified return or Sharpe, practitioners can assess whether performance relies on narrow regimes or remains robust across realistic scenarios, which is critical for regulation‑driven scenario analysis and client communication as well as diagnostics for the model's limitations. This leads us to the following question: 
\begin{quote}
\emph{Given a list of firm characteristics at time $\bar{t}$ and a decision pipeline $\pi$, which  market conditions $\vec{m}$ would result in portfolios whose predicted returns are close to a benchmark return $b$?}
\end{quote}
This question can be addressed using the following probing function
\begin{equation}\label{eq:benchmarkG}
    G(\vec{m}) = (\vec{w}_{\pi}(\vec{m})^\top \vec{r}_{\bar{t} + 1} - b)^2.
\end{equation}
This function is minimized when the realized portfolio's excess return is close to $b$. Note from Figure \ref{fig:backwardCalc} that the allocation vector $\vec{w}_\pi$ depends on the firm characteristics of the chosen assets at time $\bar{t}$. This formulation provides insight into the conditions required for the decision pipeline, $\pi$ to attain the desired level of return. To incorporate risk adjustment, the term $\vec{w}_\pi(\vec{m})^\top \vec{r}_{\bar{t} + 1}$ can be replaced by $\vec{w}_\pi(\vec{m})^\top \vec{r}_{\bar{t} + 1} - \tfrac{\lambda}{2}\vec{w}_\pi(\vec{m})^\top \vec{\Sigma} \vec{w}_\pi(\vec{m})$ where $\vec{\Sigma}$ is also obtained at time $\bar{t}$. Here we note that the covariance matrix of excess returns $\vec{\Sigma}$ at a time $\bar{t}$ is estimated using an exponentially weighted moving average as in \cite[p.31]{boyd2024markowitz}. 

\vspace{2mm}

\paragraph{Encouraging diverse portfolios.} Portfolio managers face hard diversification and concentration limits, while naive mean–variance optimization often produces unstable, highly concentrated portfolios. Probing for macro conditions that lead to diversified versus concentrated allocations can reveal the conditions under which the model predicts similar returns for assets with similar idiosyncratic risk, providing macro-level insights into how the predictive model drives allocations. It can also reveal a model's rigidness and preference for concentrated portfolios. Then, the question becomes:
\begin{quote}
\emph{Given a list of firm characteristics at time $\bar{t}$ and a decision pipeline $\pi$, which market conditions $\vec{m}$ would result in well-diversified portfolios, rather than ones concentrated in only a few assets?}
\end{quote}
The degree of diversification of a portfolio allocation can be quantified using entropy. Specifically, the entropy of an allocation weight vector $\vec{w}$ is defined as $H(\vec{w}) = - \sum_i w_i \log w_i$, which attains its lowest value when all mass is concentrated in a single asset, \textit{i.e.}, $w_{i_0} = 1$ and $w_i = 0$ for all $i \neq i_0$, and its maximum when weights are evenly distributed across assets. Accordingly, defining the probing function
\begin{equation} \label{eq:probeEntropy}
        G(\vec{m}) = -H(\vec{w}_\pi(\vec{m}))
\end{equation}
favors macroeconomic conditions that induce more diversified portfolio allocations for the decision pipeline $\pi$. 

\vspace{2mm}

\paragraph{Contrastive scenarios for decision pipelines.} Two competing models (\emph{e.g.}, PTO vs PAO) or two decision pipelines trained on different datasets will deliver differing portfolio allocations $\vec{w}_{\pi_1}(\vec{m}_{\overline{t}})$ and $\vec{w}_{\pi_2}(\vec{m}_{\overline{t}})$, which will result in distinct realized returns. Say $\pi_1$ results in a  higher excess return than $\pi_2$. To understand how the pipelines depend on the macro conditions we ask: 
\begin{quote}
    \emph{What market conditions would flip the excess returns of the two pipelines $\pi_1$ and $\pi_2$?}
\end{quote}
For this purpose, we construct the probing function (assuming $\pi_2$ has a higher return than $\pi_1$)
\[
    G_{\pi_1,\pi_2}(\vec{m}) = - (\vec{w}_{\pi_1}(\vec{m})^\top \vec{r}_{\overline{t} + 1}- \vec{w}_{\pi_2}(\vec{m})^\top \vec{r}_{\overline{t} + 1}),
\]
where $\vec{r}_{\overline{t} + 1}$ are the actual realized excess returns. This takes lower values when $\pi_1$ achieves a higher return, flipping the actual scenario. As all the firm characteristics are the same, we observe the effect of macroeconomic conditions that changes the outcomes.

As demonstrated, the framework offers significant flexibility in defining probing functions. While there are numerous ways to investigate how the models respond to macroeconomic variables, we focus on these three probing functions in our computational study.

\subsection{Theoretical Foundations of Probing.}
\label{subsec:theory}

We formalize our probing framework by viewing it as a Bayesian update over macroeconomic states, and then deriving its variational and robustness properties. Starting from a prior distribution over macro regimes, a probing question is encoded as a decision-based loss $G(\vec{m})$ that plays the role of a pseudo negative log-likelihood, yielding a Gibbs distribution $p(\vec{m})$ over macro states. We show that this distribution is equivalently the unique minimizer of an entropy-regularized expected-loss problem and that, for a fixed decision-level event, different smooth probing losses lead to Gibbs measures that concentrate on the same set of macro conditions. This perspective clarifies both the well-posedness of our construction and the sense in which our explanations are robust to the choice of probing function.

We denote by $\mathcal{M}$ the space of macro states and by $\mathcal{Y}$ the space of decision-level quantities of interest (for example, realized portfolio returns, Sharpe ratios, or vectors collecting several such performance measures). We fix a decision pipeline $\pi$, which maps macro states $\vec{m}\in\mathcal{M}$ to portfolio weights $\vec{w}_\pi(\vec{m})$. Given $\pi$, we write $Y:\mathcal{M}\to\mathcal{Y}$
for a continuous mapping that evaluates a decision-level quantity of interest (such as the realized excess return or a Sharpe-like score of pipeline $\pi$) at macro state $\vec{m}$. The probing loss $G(\vec{m})$ is defined in terms of $Y(\vec{m})$ (and hence in terms of $\pi$), and our theoretical results treat $\pi$ and $Y$ as given and study the resulting distribution $p(\vec{m})$ over macro scenarios.

\paragraph{Bayesian interpretation.} 

Let $p_0$ be a reference probability density on $\mathcal{M}$ with respect to a base measure $\mu$, which we interpret as a prior distribution over macroeconomic conditions. For example, given a date of interest $\bar{t}$, $p_0(\vec{m})$ can be taken as the probability of observing a macroeconomic condition $\vec{m}$ in a month that follows $\vec{m}_{\bar{t}}$, which we estimate e.g.\ by fitting a autoregressive model to the month to month variations of historical data.

A probing question is encoded by a nonnegative function $G:\mathcal{M}\to\mathbb{R}_+$ that penalizes macro states $\vec{m}$ where the decision pipeline $\pi$ fails to exhibit a desired behavior. Given $p_0$, $G$, and a temperature parameter $\tau>0$, we define the probing distribution as the solution of \eqref{eq:BLPforG}
where the minimization is over probability densities $p$ on $\mathcal{M}$ (with respect to $\mu$), and $\mathrm{KL}(p\|p_0)$ denotes the Kullback--Leibler divergence between $p$ and $p_0$. The objective \eqref{eq:BLPforG} balances two terms: (i) A data-like term $\mathbb{E}_p[G(m)]$ that encourages concentration on macro states with low probing loss. (ii) A complexity penalty $\mathrm{KL}(p\|p_0)$ that keeps $p$ close to the prior $p_0$. The scalar $\tau>0$ plays the role of an inverse sample size or noise level: smaller $\tau$ places more emphasis on minimizing the probing loss, while larger $\tau$ keeps the distribution closer to $p_0$.

Given the prior $p_0$ and interpreting $\frac{1}{\tau} G(\vec{m})$ as the negative log-likelihood of observing $\vec{m}$ in hypothetical scenario then the Gibbs solution to \eqref{eq:BLPforG} is the Bayes posterior, following the Bayes rule $\Prob(\vec{m} | \texttt{scenario}
) \propto \Prob(\texttt{scenario} | \vec{m} ) p_0(\vec{m})$. In this interpretation, choosing $G$ amounts to specifying a probabilistic model for our hypothetical scenario. For example, in the benchmark targeting probing scenario we probe the model by finding $\vec{m}$ which is distributed probabilistically under the assumption that the portfolio generated by the pipeline $\pi$ has an excess return which is normally distributed $\mathcal{N}(b, 2\tau)$ centered at the benchmark return $b$. The likelihood we set for the diversity encouraging scenario is less standard than a Gaussian, and is proportional to the perplexity $e^{H(\vec{w}_\pi(\vec{m}))}$ of the porfolio assignment, with portfolios with higher perplexity (higher diversification) being assigned a higher likelihood. For the last probing function where we contrast two models by flipping the returns they would give for that month, our choice of probing function is equivalent to selecting the likelihood of this scenario as log-linear on the returns, given $\vec{m}$. 


\begin{proposition}[Existence, form, and uniqueness]
    \label{prop:existence}
Assume that $\mathcal{M}$ is compact, $G$ is continuous and bounded from below, $p_0$ is a strictly positive and continuous prior density on $\mathcal{M}$ with respect to $\mu$, and the reference measure $\mu$ has full support on $\mathcal{M}$. Then, problem \eqref{eq:BLPforG} admits a unique minimizer $p_{\pi, \tau}$, which has the Gibbs density
\begin{equation}
\label{eq:gibbs}
p_{\pi,\tau}(\vec{m}) = \frac{\exp\big(-G(\vec{m})/\tau\big) p_0(\vec{m})}{\int_{\mathcal{M}}\exp\big(-G(\vec{u})/\tau\big)p_0(\vec{u})d\mu(\vec{u})}.
\end{equation}
Equivalently, $p_{\pi, \tau}(\vec{m}) \propto \exp(-\widetilde{G}(\vec{m})/\tau)$. In particular, for each fixed probing function $G$, prior $p_0$, and temperature $\tau>0$, there is a unique probing distribution $p_{\pi,\tau}$ associated with the corresponding question.
\end{proposition}

\begin{proof}
We first note that 
\[
\mathrm{KL}(p\|p_0) = \int_\mathcal{M} p(\vu)\log \frac{p(\vu)}{p_0(\vu)} d\mu(\vu) = \int_\mathcal{M} p(\vu)\log p(\vu) d\mu(\vu) - \int_\mathcal{M} p(\vu)\log p_0(\vu) d\mu(\vu).
\]
Thus, the objective in \eqref{eq:BLPforG} becomes
\[
g(p) = \int_\mathcal{M} G(\vec{m}) p(\vec{m}) d\mu(\vec{m}) + \tau \int_\mathcal{M} p(\vec{m})\log p(\vec{m}) d\mu(\vec{m}) - \tau\int_{\mathcal{M}}p(\vec{m})\log p_0(\vec{m})d\mu(\vec{m}),
\]
which can be written as
\[
g(p) = \int_\mathcal{M} \widetilde{G}(\vec{m}) p(\vec{m}) d\mu(\vec{m}) + \tau \int_\mathcal{M} p(\vec{m})\log p(\vec{m}) d\mu(\vec{m}),
\]
where $\widetilde{G}(\vec{m}) = G(\vec{m}) -\tau\log p_0(\vec{m})$. 
Since $\mathcal{M}$ is compact, $G$ is continuous, and $p_0$ is strictly positive and continuous, $\widetilde{G}$ is continuous and bounded below. Existence follows because $\mathbb{E}_p[G]$ is weakly continuous (as $G$ is continuous and bounded on the compact $\mathcal{M}$) and $\mathrm{KL}(\cdot|p_0)$ is weakly lower-semicontinuous with compact sublevel sets in the weak topology on probability measures. Uniqueness follows because $g(p)$ is strictly convex in $p$: the term $\mathbb{E}_p[G]$ is linear in $p$, while $p \mapsto \tau\int p\log p d\mu$ is strictly convex.

We next derive the optimal solution. Let $\lambda$ be the multiplier for the normalization constraint and consider the Lagrangian function
\[
L(p, \lambda) = \int_\mathcal{M} \widetilde{G}(\vec{m}) p(\vec{m}) d\mu(\vec{m}) + \tau \int_\mathcal{M} p(\vec{m})\log p(\vec{m}) d\mu(\vec{m}) + \lambda \left(\int_{\mathcal{M}} p(\vec{m})d\mu(\vec{m}) - 1\right).
\]
Formally differentiating $L$ in the direction of a perturbation $\delta p$ with $\int\delta p\,d\mu=0$ yields the first-order condition
\[
\widetilde{G}(\vec{m}) + \tau(1+\log p(\vec{m})) + \lambda = 0
\]
Thus, for some constant $C>0$, we have
\[
p(\vec{m}) = C \exp\left(-\widetilde{G}(\vec{m})/\tau\right) = C \exp\left(-G(\vec{m})/\tau\right)p_0(\vec{m})
\]
Enforcing now $\int_{\mathcal{M}} p(\vec{m})\,d\mu(\vec{m})=1$ identifies
\[
C^{-1}=\int_{\mathcal{M}}\exp(-G(\vec{u})/\tau)p_0(\vec{u}) d\mu(\vec{u}),
\]
which gives \eqref{eq:gibbs}.
\end{proof}

\paragraph{Robustness of the scenarios.} 
The discussion above allows us to state precisely which aspects are unique and which are intentionally question-dependent. At a fixed probing question, specified by a probing function $G$, prior $p_0$, and temperature $\tau$, the associated Gibbs density $p_{\pi, \tau}$ in \eqref{eq:gibbs} is unique. 
Thus, there is no ambiguity at the distribution level; different sampling schemes aim at approximating the same target. A more subtle issue is whether different probing functions that encode the \emph{same} decision event yield consistent scenario sets. To this end, we first connect the Gibbs density to conditioning on decision-level events and show how this perspective leads naturally to an event-level robustness result.

Let $Y:\mathcal{M}\to\mathcal{Y}$ be a continuous mapping that extracts a decision-level quantity of interest (e.g., realized portfolio return or a Sharpe-like score of pipeline $\pi$) from a macro state $\vec{m}$. Let $\mathcal{A}\subseteq\mathcal{Y}$ be a closed target set encoding the probing question, and define the corresponding \emph{solution set}
\begin{equation}    
\label{eq:setS}
\mathcal{S} = \{\vec{m}\in\mathcal{M}: Y(\vec{m})\in\mathcal{A}\}.
\end{equation}
In an idealized setting, we would like to work with the conditional prior $p_0(\cdot \mid Y(\vec{m})\in\mathcal{A})$, but for continuous macro variables and sharp events $\{Y(\vec{m})\in\mathcal{A}\}$, this conditional distribution is often ill-defined in the usual measure-theoretic sense. Therefore, we obtain a smooth approximation by defining probing as the distance to $\mathcal{A}$. The following proposition formalizes our approach.

\begin{proposition}[Approximate Conditioning]
\label{prop:robustness}
Let $\mathcal{M}$ be compact, $Y:\mathcal{M}\to\mathcal{Y}$ be continuous, and $\mathcal{A}\subseteq\mathcal{Y}$ be closed with nonempty preimage $\mathcal{S}$ given in \eqref{eq:setS}, and $p_0$ be a strictly positive and continuous prior density on $\mathcal{M}$. Consider any probing function $G:\mathcal{M}\to\mathbb{R}_+$ satisfying the following conditions:
\begin{enumerate}
    \item $G(\vec{m})\ge 0$ for all $\vec{m}\in\mathcal{M}$ and $G(\vec{m})=0$ if and only if $\vec{m}\in\mathcal{S}$,
    \item $G$ is continuous on $\mathcal{M}$.
\end{enumerate}
For $\tau>0$, let $p_\tau(\vec{m})$ be of the form \eqref{eq:gibbs}. Then, for every open neighborhood $U$ of $\mathcal{S}$ in $\mathcal{M}$, we have 
\[
\lim_{\tau\downarrow 0} \int_U p_\tau(\vec{m})\,d\mu(\vec{m}) = 1.
\]
In particular, if we take the canonical choice 
\[
G(\vec{m}) = \mathrm{dist}(Y(\vec{m}),\mathcal{A})^2, 
\]
where $\mathrm{dist}(y,\mathcal{A}) = \inf_{a\in\mathcal{A}}\|y-a\|$, then the corresponding Gibbs densities concentrate on $\mathcal{S}$ as $\tau\to 0$ and $p_\tau$ converges weakly to the conditional prior $p_0(\cdot\mid Y(\vec{m})\in\mathcal{A})$.
\end{proposition}

\begin{proof}
Let $U$ be an open neighborhood of $\mathcal{S}$. Then $\mathcal{M}\setminus U$ is compact and disjoint from $\mathcal{S}$. On $\mathcal{M}\setminus U$, we have $G(\vec{m})>0$ by the assumption that the zero set of $G$ is exactly $\mathcal{S}$. By continuity of $G$ and compactness, there exists $\delta>0$ such that
\[
G(\vec{m}) \ge \delta \text{ for all } \vec{m}\in\mathcal{M}\setminus U.
\]
On the other hand, $G(\vec{m})=0$ on $\mathcal{S}$, so by continuity for all $r > 0$ there exists a measurable set $V \subset U$ with $\mu(V)>0$ such that
\[
G(\vec{m}) \le r \text{ for all } \vec{m}\in V.
\]
Since $p_0$ is strictly positive and continuous on the compact set $\mathcal{M}$, there exist constants $0 < c \le C < \infty$ such that $c \le p_0(\vec{m}) \le C$ for all $\vec{m}\in\mathcal{M}$. Using these bounds, we obtain 
\[
\int_{\mathcal{M}} \exp(-G(\vec{u})/\tau)p_0(\vec{u}) d\mu(\vec{u}) \ge \int_{V} \exp(-G(\vec{u})/\tau)p_0(\vec{u}) d\mu(\vec{u}) \ge c\exp(-r/\tau)\mu(V),
\]
and 
\[
\int_{\mathcal{M}\setminus U} \exp(-G(\vec{u})/\tau) p_0(\vec{u}) d\mu(\vec{u}) \le C\exp(-\delta/\tau) \mu(\mathcal{M}\setminus U).
\]
Therefore, 
\[
\int_{\mathcal{M}\setminus U} p_\tau(\vec{m}) d\mu(\vec{m}) = \frac{\int_{\mathcal{M}\setminus U} \exp(-G(\vec{u})/\tau) p_0(\vec{u}) d\mu(\vec{u})}{\int_{\mathcal{M}} \exp(-G(\vec{u})/\tau) p_0(\vec{u})d\mu(\vec{u})} \le \frac{C\mu(\mathcal{M}\setminus U)}{c \mu(V)}\exp(-(\delta-r)/\tau),
\]
which converges to zero as $\tau\downarrow 0$ whenever $r < \delta$. Hence $\int_U p_\tau(\vec{m})d\mu(\vec{m})\to 1$ as $\tau\downarrow 0$.

For the conditional prior statement: as $\tau\downarrow 0$, the factor $\exp(-G(\vec{m})/\tau)$ converges to $\vec{1}_{\mathcal{S}}(\vec{m})$ point-wise, so $p_\tau d\mu$ places all its mass on $\mathcal{S}$ with weights proportional to $p_0$, which is precisely $p_0(\cdot\mid Y(\vec{m})\in\mathcal{A})$.
\end{proof}

\begin{corollary}[Event-level Uniqueness]
\label{cor:unique}
Under the assumptions of Proposition \eqref{prop:robustness}, consider two probing functions that satisfy the two conditions of the proposition and have the same zero set $\mathcal{S}$. Then, both families of Gibbs densities concentrate on the same set $\mathcal{S}$ as $\tau\to 0$.
\end{corollary}

This combined approximate conditioning and event-level uniqueness results show that robustness should be evaluated at the level of the \emph{conditioning event}. For a fixed event $\{Y(\vec{m})\in\mathcal{A}\}$, any two well-behaved probing functions that vanish exactly on the same solution set $\mathcal{S}$ induce Gibbs densities whose supports coincide asymptotically as $\tau\to 0$. In other words, within a given probing question (fixed event and prior), the \emph{scenario set} highlighted by probing is unique up to smoothing. Non-uniqueness arises only when one changes the event itself, that is, when $\mathcal{A}$ changes; for example, from ``target return'' to ``target Sharpe'', or from ``high diversification'' to ``concentrated allocations''. Such changes correspond to asking substantively different economic questions, and it is appropriate that they yield different explanatory macro regimes.

We note that the event-level uniqueness result in Corollary \ref{cor:unique} is an asymptotic statement: it characterizes how Gibbs densities behave as $\tau\to 0$ and shows that, in this limit, different probing functions for the same event concentrate on the same solution set $\mathcal{S}$. In applications, however, we do not drive $\tau$ to zero. Instead, $\tau$ is treated as a smoothing and regularization parameter which may require tuning. In practice, we recommend choosing $\tau$ to balance the trade-off between concentrating on low-loss regions and maintaining a sufficiently broad distribution for sampling. The robustness result implies that, as long as $\tau$ is not too large, different probing functions for the same event should yield qualitatively similar scenario sets, even if they differ in their smoothing properties.

Taken together, the Bayesian formulation, the variational characterization, and the event-level uniqueness result show that probing is a well-posed and interpretable construction: for each fixed decision pipeline $\pi$, each probing question defines a unique Gibbs distribution $p_\tau(\vec{m})$ over macro states that can be viewed as an entropy-regularized solution. The support of this solution is robust to the choice of probing loss at the event level, and its sharpness around the event can be tuned through the temperature parameter $\tau$. The practitioner can tune $\tau$ to be large enough so that the generated data points conform (up to an acceptable degree of error) to the desired scenario for majority of the cases. For example, if we apply the benchmark return probing function \eqref{eq:benchmarkG}, then one can decrease the temperature parameter $\tau$ until the generated scenarios produce returns that lie within a reasonable variance around $b$. The limits of what counts of ``reasonable'' changes according to the quantity under constraint and the kind of behavior we want to extract out of the model.

\section{Computational Study.}
\label{sec:NumEx}
This section presents an empirical evaluation of the PTO and PAO training strategies, together with the proposed probing framework. We first describe the dataset construction and experimental design. We then compare out-of-sample portfolio performance across the two training strategies under several allocation configurations. Finally, we use the probing framework to generate macroeconomic scenarios that answer the decision-level questions introduced in Section~\ref{subsec:probingfunctions}, showing how the framework can diagnose and compare the behavior of trained portfolio decision pipelines.

\subsection{Data and Experimental Design.}
\label{subsec:design}
Our empirical analysis is based on a large-scale monthly panel of firm-level and macroeconomic data, which we use to train and evaluate the PTO and PAO training strategies. We construct a monthly firm--month panel dataset by combining firm-level characteristics with realized stock returns, following standard empirical asset pricing practices \citep{GuKellyXiu2020}.

\paragraph{Firm characteristics.} We obtain monthly firm-level predictors from the open-source cross-sectional asset pricing dataset compiled by \citet{ChenZimmermann2022}, which contains a large set of economically interpretable firm characteristics. Following standard data quality procedures, we exclude predictors with more than 70\% missing values over the full sample, resulting in a final set of $C = 140$ characteristics. The remaining predictors are preprocessed on a month-by-month basis. Missing values are imputed using the cross-sectional median, with a fallback of zero for characteristics that are entirely missing in a given month. We winsorize each characteristic at the 1st and 99th percentiles to limit the influence of extreme observations, and then apply cross-sectional rank normalization to map values into the interval $[-1,1]$. This transformation represents each firm by its relative position in the monthly cross-section and is commonly used in machine-learning-based return prediction to improve numerical stability \citep{GuKellyXiu2020}.

\paragraph{Macroeconomic variables.} We obtain macroeconomic conditioning information from the Goyal--Welch monthly predictor dataset \citep{GoyalWelch2008,GoyalWelchData2024}, a widely used benchmark in return forecasting. From this source, we construct $M=9$ macroeconomic predictors: dividend--price ratio (\texttt{dp}), earnings--price ratio (\texttt{ep}), book-to-market ratio (\texttt{bm}), net equity expansion (\texttt{ntis}), Treasury bill rate (\texttt{tbl}), term spread (\texttt{tms}), default spread (\texttt{dfy}), stock market variance (\texttt{svar}), and inflation rate (\texttt{infl}). Definitions of these variables are provided in Appendix~\ref{app:macro_preds}.
 
\paragraph{Feature construction.} Following Section~\ref{sec:Method}, we construct firm--macro interaction features by combining the $C$ firm-level characteristics $\mathbf{c}_{a,t}$ with their interactions with the $M$ macroeconomic variables $\mathbf{m}_t$, as illustrated in Figure \ref{fig:training_process}. This gives a total of $F = (M + 1) \times C = 1{,}400$ features per firm--month observation. The resulting interaction structure enables the predictive model to capture how the cross-sectional relationship between firm characteristics and expected returns varies with macroeconomic conditions \citep{GuKellyXiu2020}.

\paragraph{Returns and label cleaning.} We define the prediction target as the one-month-ahead realized excess return, computed as the raw return minus the risk-free rate. Following the convention established in Section~\ref{sec:Method}, we denote excess returns simply as $r_{a,t+1}$ throughout; hence, $\mathbf{r}_{t+1} \in \mathbb{R}^A$ collects the cross-section of excess returns at time $t+1$. Observations with missing excess returns are removed from the sample. To stabilize training and mitigate the influence of extreme values, we cap large negative excess returns at $-99\%$ and winsorize the upper tail using a high-quantile cutoff.

\paragraph{Sample period and split.}
Our empirical analysis spans the period from January 1980 through November 2024, comprising $T = 539$ months. To preserve temporal integrity and avoid look-ahead bias, we adopt a strictly chronological split into training, validation, and test sets, following established practice in empirical asset pricing studies \citep{GuKellyXiu2020,Chen24}:
\begin{itemize}
    \item \textit{Training set}: January 1980 -- December 2005 (312 months),
    \item \textit{Validation set}: January 2006 -- December 2015 (120 months),
    \item \textit{Test set}: January 2016 -- November 2024 (107 months).
\end{itemize}
Table~\ref{tab:data_summary} provides descriptive statistics for each subset, including the number of firm--month observations, average monthly cross-sectional coverage, and firm-level excess-return moments in each sample period.

\begin{table}[htbp]
\centering
\caption{Dataset summary statistics.}
\label{tab:data_summary}
\begin{tabular}{lcccc}
\toprule
& \textbf{Train} & \textbf{Validation} & \textbf{Test} & \textbf{Full sample} \\
\midrule
Period & 1980--2005 & 2006--2015 & 2016--2024 & 1980--2024 \\
Months & 312 & 120 & 107 & 539 \\
Firm--months & 2{,}204{,}243 & 815{,}813 & 880{,}204 & 3{,}900{,}260 \\
Avg.\ firms/month & 7{,}065 & 6{,}798 & 8{,}226 & 7{,}236 \\
\midrule
\multicolumn{5}{l}{\textit{Excess return statistics (firm--month, monthly)}} \\
Mean & 0.41\% & 0.52\% & 0.89\% & 0.54\% \\
Std.\ dev. & 16.4\% & 14.2\% & 15.8\% & 15.3\% \\
\bottomrule
\end{tabular}
\end{table}

\paragraph{Predictive architecture.}
Both the PTO and PAO training strategies rely on the same feedforward neural network predictor $f_{\boldsymbol{\theta}}:\mathbb{R}^{F}\to\mathbb{R}$ that maps firm--macro interaction features to one-month-ahead predicted excess returns. We use a compact multilayer perceptron with three hidden layers of sizes $(32,16,8)$, ReLU activations, batch normalization, and dropout regularization with rate $0.5$. This design is consistent with standard nonlinear prediction architectures used in empirical asset pricing \citep{GuKellyXiu2020, Chen24}.

\paragraph{Portfolio allocation (common across PTO and PAO).}
At each decision month $\bar{t}$, the predictor outputs one-month-ahead return forecasts $\widehat{r}_{a,t+1}$ for all firms $a = 1, \ldots, A$ observed at time $t$. We work with a curated universe of $500$ firms, constructed by anchoring on the constituents of the S\&P~100 and adding large-capitalization firms to ensure broad industry coverage. This provides a liquid and economically representative cross-section over the sample. At each decision month, we retain firms with a complete trailing return history of 60 months, which is required for risk estimation. 

\paragraph{Risk estimation and allocation configurations.}
Portfolio risk is estimated using an exponentially weighted moving average (EWMA) covariance matrix computed from realized excess returns over the past $60$ months, with a decay factor of $0.94$, consistent with standard RiskMetrics-style smoothing \citep{Lau2015NIGVaR}. To improve numerical stability, we apply diagonal shrinkage with intensity $0.10$ and add a ridge stabilization term $10^{-6}$, following best practices in covariance regularization \citep{LedoitWolf2004}.

Given the predicted returns $\widehat{\mathbf{r}}_{t+1}$ and the estimated covariance matrix $\boldsymbol{\Sigma}$, portfolio weights are obtained by solving the robust mean--variance allocation problem~(1) introduced in Section~\ref{sec:Method}. We evaluate both the baseline mean--variance allocation ($\kappa=0$) and mean-uncertainty-penalized variants ($\kappa>0$). In our experiments, we consider two uncertainty scalings, $\boldsymbol{\Omega}=\mathrm{diag}(\boldsymbol{\Sigma})$ and $\boldsymbol{\Omega}=\mathbf{I}$, following the practical recommendations of \citet{yin2021practical}.

\subsection{Training Process: Model Construction.}
\label{subsec:TrainingResults}
In this section, we evaluate the PTO and PAO training strategies using the same decision pipeline components: universe construction, risk estimation, allocation problem, constraints, and monthly rebalancing rules as described above. The two strategies differ only in how the neural network parameters $\boldsymbol{\theta}$ are optimized (see Figure~\ref{fig:training_process} in Section~\ref{sec:Method}). We next summarize the key implementation details for each training strategy.

\paragraph{Implementation details.}
In the PTO training strategy, the predictor is trained for pointwise accuracy by minimizing mean squared error between predicted and realized excess returns on the training set, as described in Section~\ref{sec:Method}. After training, the predictor parameters $\boldsymbol{\theta}^*$ are held fixed. Out of sample, each month we predict returns for all firms, estimate risk via EWMA covariance, and solve the allocation problem for the chosen $(\lambda,\kappa,\boldsymbol{\Omega})$ configuration.

In the PAO training strategy, the predictor is trained jointly with the downstream allocation rule so that model parameters are optimized for portfolio performance rather than pointwise forecast accuracy \citep{Donti17,Elmachtoub2022}. Training proceeds by backpropagating through a differentiable convex optimization layer implemented via \texttt{CVXPY} and \texttt{cvxpylayers} \citep{DiamondBoyd2016,Agrawal2019}. Within this decision-focused training strategy, we consider three portfolio-level training objectives corresponding to the score function $S(\mathbf{w}_\pi, \ldots)$ in Figure~\ref{fig:training_process}: maximizing realized portfolio return, maximizing mean--variance utility, and maximizing the annualized Sharpe ratio. Model selection is conducted on the validation period using the corresponding portfolio criterion, and final evaluation is performed strictly out of sample on the test period.

To explore the interaction between learning objectives and allocation preferences, we evaluate a grid of hyperparameters. Risk aversion is set to $\lambda \in \{5,10,20\}$, the mean-uncertainty penalty to $\kappa \in \{0.0,0.1,0.5,1.0,10.0\}$, and uncertainty scaling to $\boldsymbol{\Omega} \in \{\mathrm{diag}(\boldsymbol{\Sigma}), \mathbf{I}\}$. For PAO, we additionally vary the training loss across three objectives: return, utility, and Sharpe ratio. Full experimental details are reported in Appendix~\ref{app:training_details} (Table~\ref{tab:experimental_setup}).

\paragraph{Performance results.} We evaluate out-of-sample (OOS) performance using annualized mean excess return, annualized volatility, and annualized Sharpe ratio (scaled by $\sqrt{12}$) over the test period January~2016--November~2024, comprising 107 months.  Detailed summary statistics and full parameter sweeps are reported in Appendix~\ref{app:training_details}. Here, we focus on time-series diagnostics that reveal when and how performance differences arise. As a baseline, we include an equal-weighted portfolio formed on the same monthly universe.

We begin by evaluating the PTO training strategy, in which the return predictor is trained independently of the allocation rule and subsequently deployed within a robust mean--variance portfolio optimizer. Monthly cross-sectional forecasts of one-month-ahead excess returns are generated using the feedforward network described in Section~\ref{subsec:design}. These forecasts are then translated into portfolio weights $\mathbf{w}_\pi$ via long-only robust mean--variance optimization with full-investment constraints. In terms of predictive accuracy, the FNN model achieves an out-of-sample $R^2$ of approximately $0.35\%$ for excess return forecasts relative to a zero benchmark on the test period. This magnitude is consistent with the modest but economically meaningful predictability typically documented in the empirical asset-pricing literature \citep{GuKellyXiu2020, Chen24}. In PAO training, by contrast, the network parameters $\boldsymbol{\theta}$ are learned jointly through the allocation layer, while holding the data, universe formation, risk estimation, and feasibility constraints fixed. 

Figures~\ref{fig:agg_wealth} and \ref{fig:agg_roll_sharpe} illustrate the dynamic performance of representative pipelines. The selected configurations include the equal-weight benchmark on the same investible universe, a strong PTO configuration (robust mean–variance with $\boldsymbol{\Omega} = \mathrm{diag}(\boldsymbol{\Sigma})$ and moderate $\kappa$), and two PAO pipelines trained end-to-end using return and Sharpe-ratio objectives, respectively. Figure~\ref{fig:agg_wealth} compares cumulative wealth paths. Both PTO and PAO generate substantial long-run compounding gains relative to the equal-weight benchmark, with separation emerging over multi-year horizons rather than from a small number of isolated months. 

\begin{figure}[h]
\centering
\includegraphics[width=0.95\linewidth]{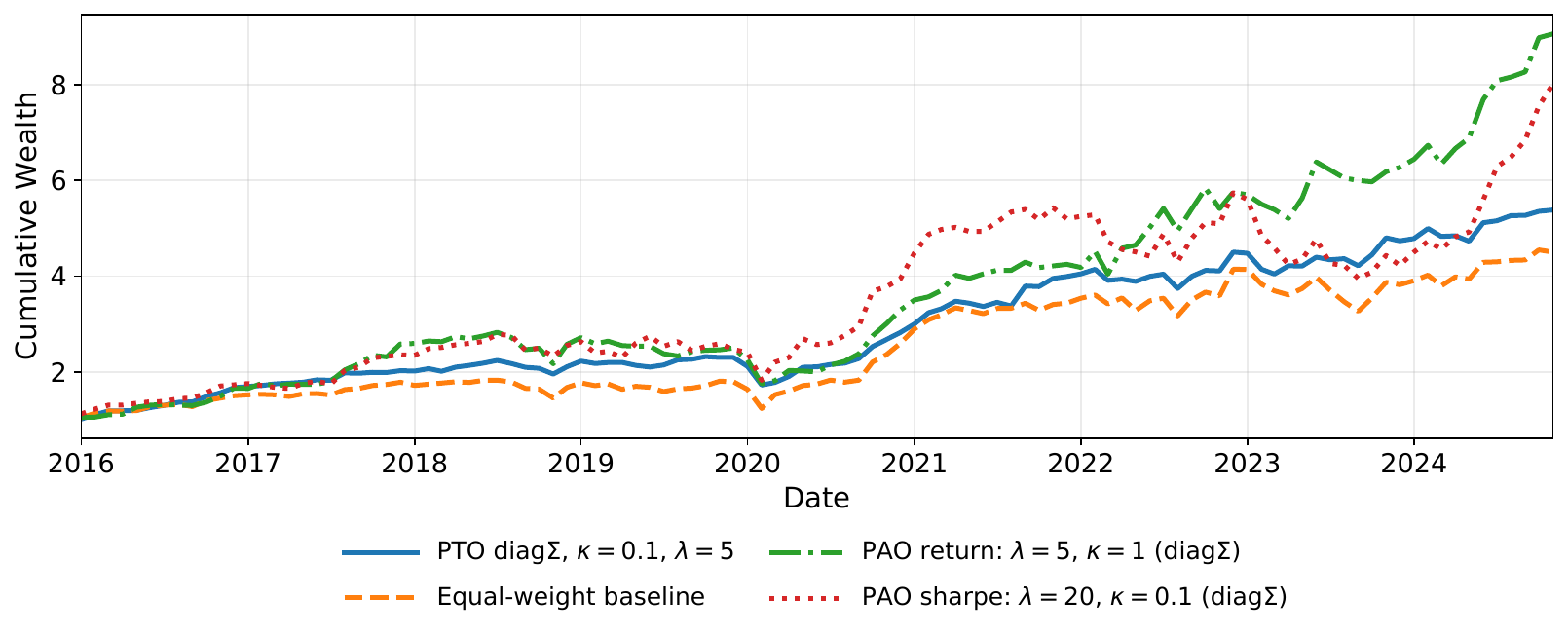}
\caption{Aggregate cumulative wealth paths over the test period (2016--2024) for the equal-weight benchmark and selected PTO/PAO configurations.}
\label{fig:agg_wealth}
\end{figure}

\begin{figure}[h]
\centering
\includegraphics[width=0.95\linewidth]{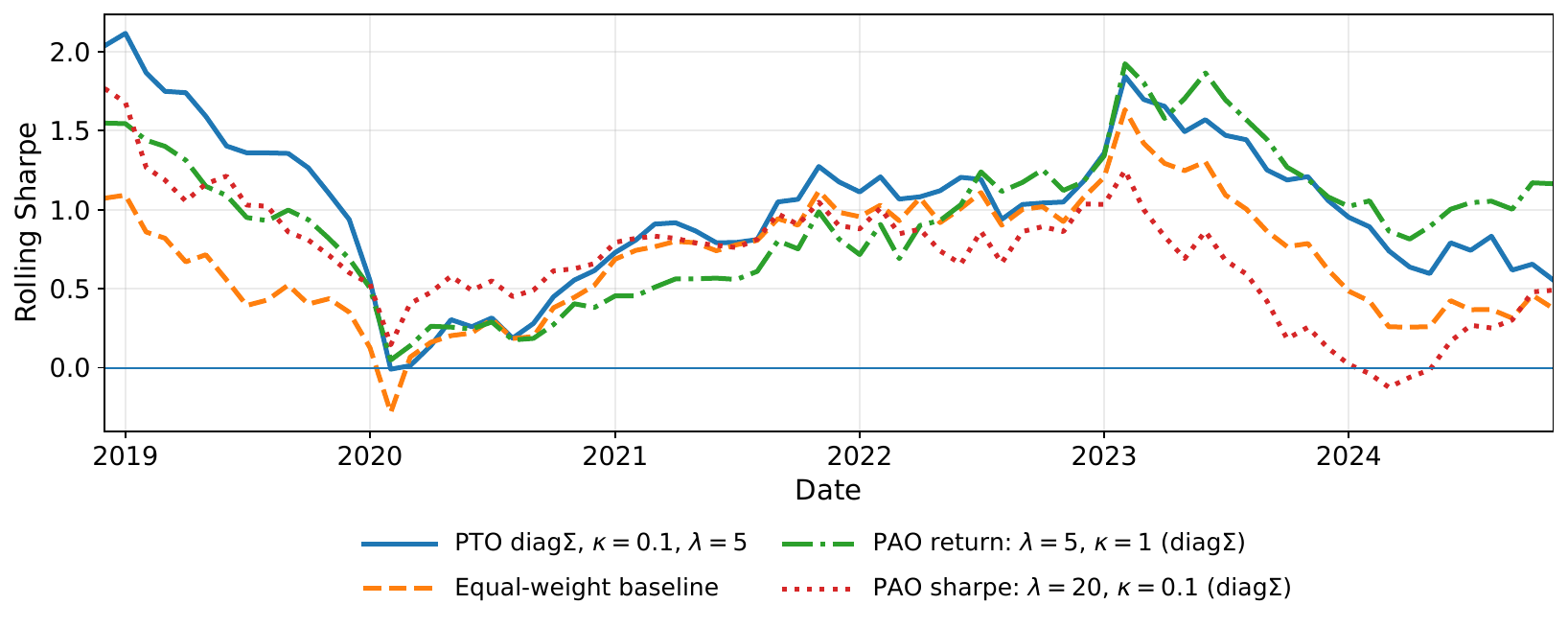}
\caption{Aggregate rolling Sharpe ratios (annualized; 36-month window) over the test period for the equal-weight benchmark and selected PTO/PAO configurations.}
\label{fig:agg_roll_sharpe}
\end{figure}

\subsection{Generation Process: Model Probing.}
\label{subsec:results}

We now evaluate the proposed probing framework. The purpose of this stage is not to assess return forecasts in isolation, but to explain how trained portfolio decision pipelines translate changes in macroeconomic inputs into optimized portfolio decisions. In each experiment, we perturb the macroeconomic state $\vec{m}$, reconstruct the corresponding firm--macro interaction features, pass the resulting features through a fixed trained decision pipeline, and solve the same robust long-only portfolio optimization problem. The generated scenarios therefore reveal how portfolio decisions respond to changes in the macroeconomic inputs entering the pipeline.

\paragraph{Scenario construction and representative states.}
\label{par:scenario_representative_states}

All probing experiments follow the same computational protocol. For a fixed trained pipeline $\pi$ and a standardized macroeconomic state $\vec{m}\in\mathbb{R}^{9}$, with coordinates $(dp,ep,bm,ntis,tbl,tms,dfy,svar,infl)$, we rebuild the firm--macro interaction features at $\vec{m}$, obtain the pipeline-specific score vector $\vec{b}_{\pi}(\vec{m})$, and solve the robust long-only portfolio problem

\begin{equation}
\label{eq:runtime_policy_map}
    \vec{w}_{\pi}(\vec{m})
    \in
    \operatorname*{arg\,max}_{\vec{w}}
    \left\{
    \vec{b}_{\pi}(\vec{m})^\top \vec{w}
    -
    \kappa \sqrt{\vec{w}^{\top}\Omega\vec{w}}
    -
    \frac{\lambda}{2}\vec{w}^{\top}\Sigma\vec{w}
    \quad
    \text{s.t.}
    \quad
    \mathbf{1}^{\top}\vec{w}=1,\quad
    \vec{w}\geq 0
    \right\}.
\end{equation}

For a given anchor month $\bar{t}$, the anchor-date firm cross-section, the realized next-month return vector, and the covariance estimation window are held fixed. The scenario generator perturbs only the macroeconomic state, isolating the role of the macroeconomic inputs in the decision rule itself.

Each experiment defines a decision-level probing function $G(\vec{m})$. Economic plausibility is encoded separately by the prior $p_{0}$, defined as the stationary law of the estimated VAR$(1)$ macroeconomic process in Section~\ref{subsec:generation}). Thus, the probing functions used below are pure decision-level objectives. We sample the resulting Gibbs distribution $p_{\pi, \tau}(\vec{m})\propto\exp(-\widetilde{G}(\vec{m})/\tau)$ using the Metropolis-adjusted Langevin algorithm. Each experiment uses four independent chains of $\,12,000\,$ steps. The first half of each chain is discarded as burn-in, and we apply thinning to the post-burn-in draws to reduce autocorrelation. We summarize each chain by a \emph{chain-representative macroeconomic state}, defined as the average of the post-burn-in thinned draws from that chain. Regime classifications, historical-analogues, and recession-probability summaries are then applied as interpretive diagnostics to these generated states; they are not part of the probing objective itself (Appendix~\ref{app:econ_interp}). 

\paragraph{Predict-then-optimize versus predict-and-optimize: closing the gap.}
\label{par:pto_pao_catchup}
Our first probe compares the two training paradigms at an anchor date where the fixed PAO pipeline outperforms the PTO pipeline. We use June~2024 as the anchor because it is a recent test-period month in which fixed PAO has a clear realized advantage over PTO, while remaining outside an acute crisis episode. This allows us to examine whether the PAO advantage persists when the macroeconomic state is perturbed within a plausible neighbourhood of the anchor. At the observed June~2024 anchor state, the fixed PAO pipeline realizes an excess return of $0.33\%$, whereas the PTO pipeline realizes $-1.12\%$, implying a return gap of $1.45$ percentage points in favor of fixed PAO. We therefore ask:
\begin{quote}
\emph{Under which macroeconomic states does the PTO pipeline close or reverse its realized return
gap relative to the fixed PAO pipeline?}
\end{quote}
With $r_{j}(\vec{m})=\vec{w}_{j}(\vec{m})^{\top}\vec{r}_{\bar t+1}$ for $j\in\{\mathrm{PTO},\mathrm{PAO}\}$, the probe uses the decision-level objective
\begin{equation}
\label{eq:pto_pao_probe}
G_{\mathrm{PTO>PAO}}(\vec{m}) = -\bigl(r_{\mathrm{PTO}}(\vec{m})-r_{\mathrm{PAO}}(\vec{m})\bigr),
\end{equation}
so that low values of $G$ correspond to states under which PTO overtakes fixed PAO. Economic plausibility is supplied by the prior $p_{0}$ localized around the June~2024 anchor.

The ranking reverses in all four chains. At the generated states, PTO pipeline outperforms fixed PAO by $1.4$, $2.5$, $2.5$, and $3.7$ percentage points, with a median improvement of $+2.5$ percentage points. Hence, the observed $1.45$-point advantage of fixed PAO at the anchor is not only eliminated but reversed under the generated macroeconomic states.

Figure~\ref{fig:pto_pao_density} shows how the macroeconomic inputs change in the states where PTO catches up. The clearest shift occurs in the default (credit) spread $dfy$, which widens relative to the anchor. This is accompanied by an increase in stock-market variance $svar$ and a rise in the dividend--price ratio $dp$, indicating cheaper equity valuations. Taken together, these movements describe a coherent macro-financial environment characterized by tighter credit conditions, higher market uncertainty, and more defensive pricing rather than a set of unrelated perturbations. Figure~\ref{fig:pto_pao_regime} provides a regime-based interpretation of these generated states. Relative to the June~2024 anchor, which is classified almost entirely as expansion, the generated chains are reclassified as contraction. This diagnostic should not be interpreted as evidence of a structural regime switch. Rather, it provides a compact summary of the fact that the generated macroeconomic states jointly display weaker financial conditions than the anchor. Figure~\ref{fig:pto_pao_nfci} provides additional historical context. The nearest historical analogues to the generated states are concentrated in late~2022 and 2023, a period marked by higher interest rates, an inverted yield curve, and 2023 regional-banking stress. Since these episodes were not formal recession months, the generated states are best interpreted as plausible credit-tightening or contraction-like environments rather than crash scenarios.

Taken together, the results indicate that the superiority of PAO at the June~2024 anchor is macro state dependent. While the decision-focused learning literature generally finds PAO pipelines to outperform the two-step PTO baseline \citep{Elmachtoub2022, Uysal24}, our probe shows this superiority is not uniform across macroeconomic conditions. Under more adverse macro-financial conditions, PTO erases and reverses the observed return gap. The PAO advantage at this anchor is therefore a fair-weather one: it holds when markets are calm but not when they turn stressed. Overall, this experiment illustrates how the probing framework identifies macroeconomic states under which the relative performance of competing decision pipelines changes.

\begin{figure}[htbp]\centering
\includegraphics[width=0.95\linewidth]{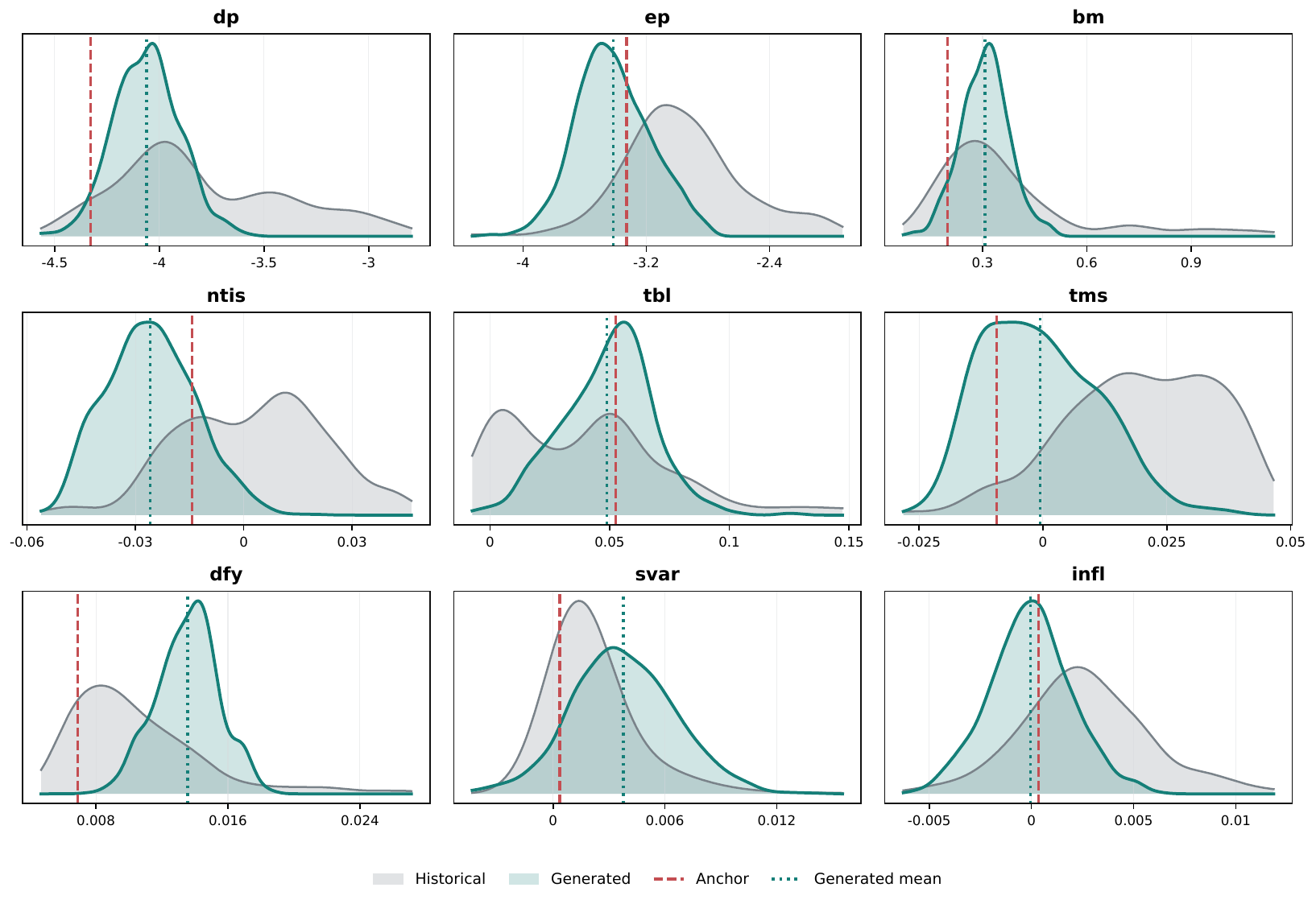}
\caption{PTO--PAO catch-up: generated macroeconomic density against the historical
panel.}
\label{fig:pto_pao_density}
\end{figure}

\begin{figure}[htbp]\centering
\includegraphics[width=0.78\linewidth]{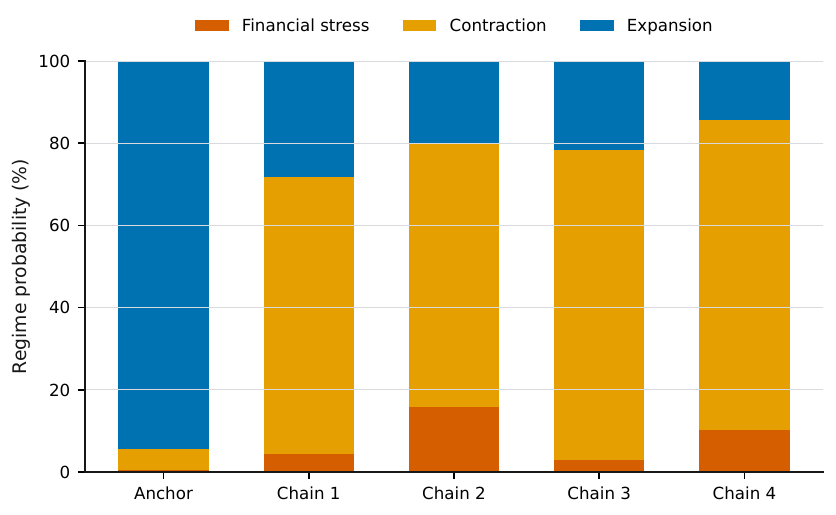}
\caption{PTO--PAO catch-up: regime-probability composition of the anchor and generated states.}
\label{fig:pto_pao_regime}
\end{figure}

\begin{figure}[htbp]\centering
\includegraphics[width=0.95\linewidth]{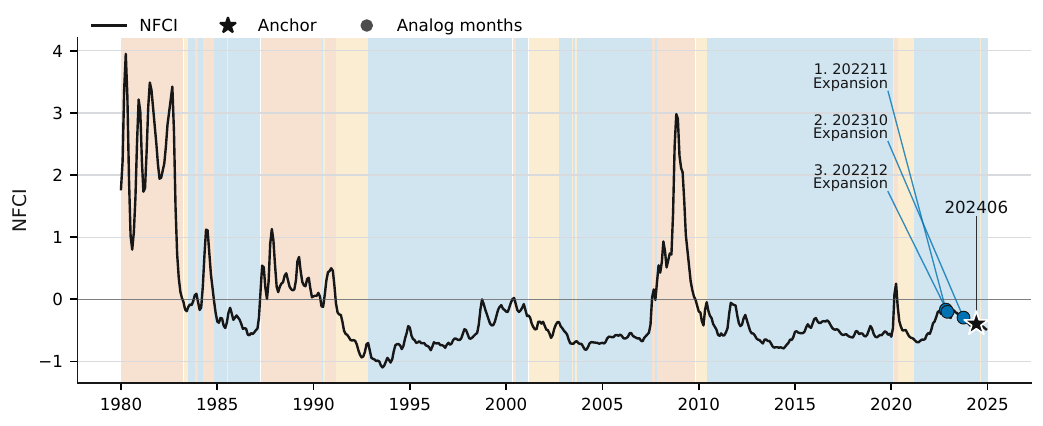}
\caption{PTO--PAO catch-up: NFCI history with the nearest generated-state analogues.}
\label{fig:pto_pao_nfci}
\end{figure}

\paragraph{Allocation structure of the fixed PAO pipeline: diversification.}
\label{par:pao_diversify}
We next turn to a single fixed pipeline and examine how macroeconomic inputs affect its allocation \emph{structure}. Because the PAO learning objective is evaluated after the optimization layer, the resulting pipeline may place most of its weight on a small number of assets when the learned scores interact strongly with the robust allocation problem. We use April~2020 as the anchor because it corresponds to an acute COVID-19 stress episode in the test period and produces a relatively concentrated fixed PAO allocation. It therefore provides a natural setting for probing when the same trained pipeline spreads its allocation more broadly. This motivates the following probing question:
\begin{quote}
\emph{Under which macroeconomic states does the fixed PAO pipeline spread its allocation more broadly across the eligible assets?}
\end{quote}

We measure allocation dispersion using entropy $H(\vec{w})=-\sum_{i}w_{i}\log w_{i}$, the Herfindahl concentration index $\mathrm{HHI}(\vec{w})=\sum_{i}w_{i}^{2}$, and the effective number of holdings $N_{\mathrm{eff}}=1/\mathrm{HHI}$. The diversification probe is defined as
\begin{equation}
\label{eq:pao_diversify_probe}
G_{\mathrm{div}}(\vec{m}) = -H\bigl(\vec{w}_{\mathrm{PAO}}(\vec{m})\bigr),
\end{equation}
so that low values of $G_{\mathrm{div}}(\vec{m})$ correspond to macroeconomic states under which the fixed PAO allocation has higher entropy and is therefore more diversified. Economic plausibility is supplied by the prior $p_{0}$, localized around the April~2020 anchor.

At the April~2020 anchor, the fixed PAO allocation is relatively concentrated, with entropy $2.27$, an effective number of holdings of $5.6$, and a maximum weight of $36.7\%$. In the four reported chains, the generated states increase allocation entropy to the range $2.31$--$2.64$ and the effective number of holdings to $5.8$--$9.4$. The HHI and maximum portfolio weight fall correspondingly. These results show that the concentration of the fixed PAO allocation is state-dependent: under the generated macroeconomic states, the same trained pipeline spreads the portfolio more broadly across the eligible assets. The probe therefore explains an allocation mechanism, rather than only a return outcome.

Figure~\ref{fig:pao_diversify_density} summarizes the macroeconomic states associated with broader PAO allocations. The main financial-stress indicators ease relative to the April~2020 anchor: stock-market variance $svar$ falls and the default spread $dfy$ narrows. These movements indicate calmer financial-market conditions and lower credit stress. Under these macroeconomic states, the learned scores no longer induce the robust optimizer to concentrate heavily in a small subset of assets. The allocation therefore becomes more dispersed, as reflected by higher entropy, a larger effective number of holdings, and a lower maximum portfolio weight.

Figures~\ref{fig:pao_diversify_regime} and~\ref{fig:pao_diversify_nfci} provide complementary interpretation. The nearest historical analogues in Figure~\ref{fig:pao_diversify_nfci} are calm months, most closely October~2005,which is consistent with the reduction in market-stress indicators. In line with this, the regime classifier in Figure~\ref{fig:pao_diversify_regime} moves the generated states away from the April~2020 financial-stress anchor and labels all four chains as expansion. The diversification probe therefore identifies a macroeconomic region in which financial-market stress is substantially lower than at the anchor and conditions have normalized away from acute stress. The broader allocation should accordingly be interpreted as being associated with calmer financial conditions as markets recover, rather than with the stressed regime of the anchor itself.

Overall, the diversification probe shows that the concentration of the fixed PAO allocation is not an invariant property of the trained model. At the April~2020 stress anchor, the pipeline assigns relatively large weights to a small number of assets. When financial-market stress eases, the same pipeline produces a more diversified allocation. This illustrates how the probing framework can recover a decision-level allocation mechanism from a trained portfolio pipeline without modifying the model or imposing an explicit diversification constraint.

\begin{figure}[h]\centering
\includegraphics[width=0.95\linewidth]{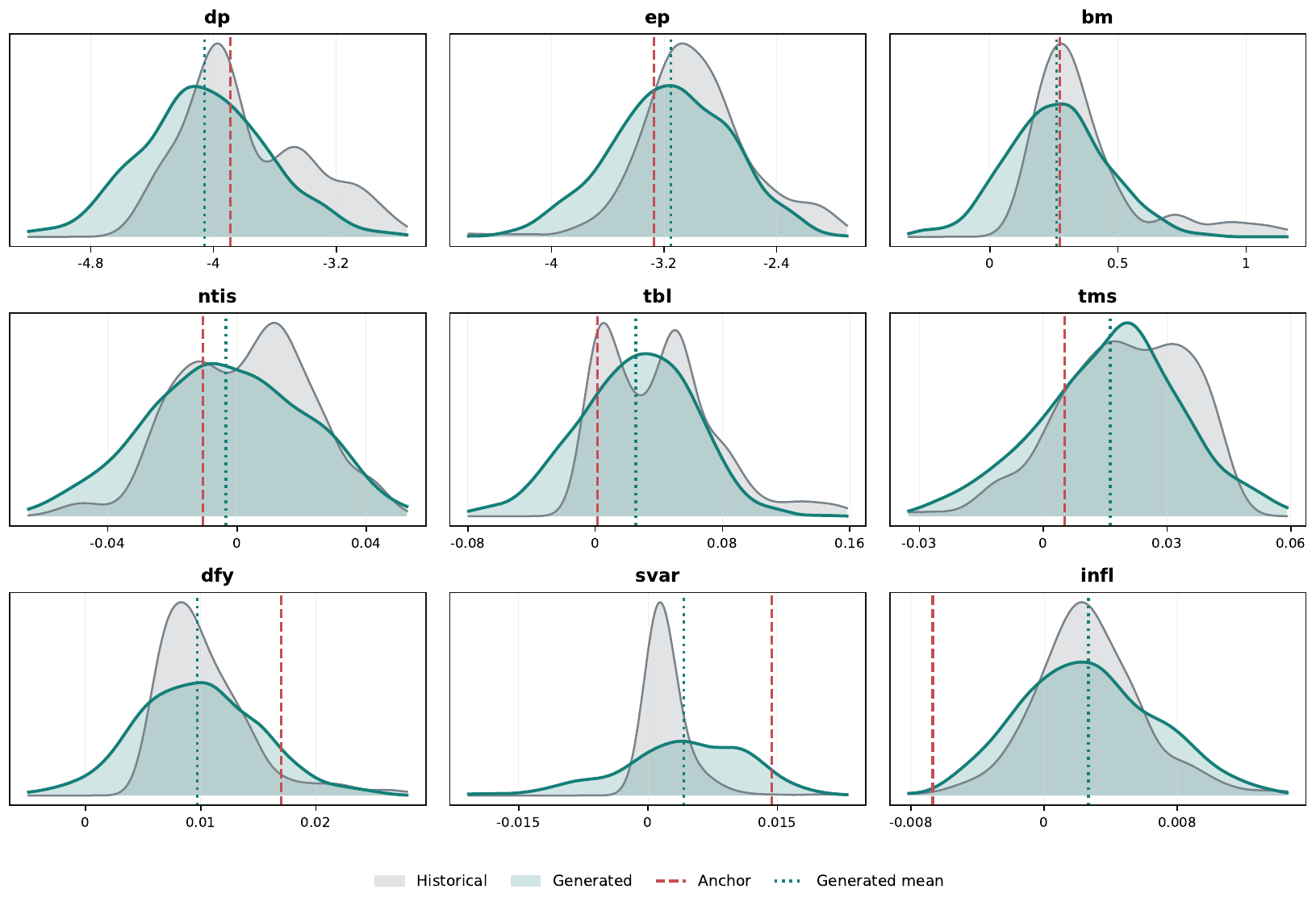}
\caption{Fixed PAO diversification: generated macroeconomic states relative to the historical panel.}
\label{fig:pao_diversify_density}
\end{figure}

\begin{figure}[h]\centering
\includegraphics[width=0.72\linewidth]{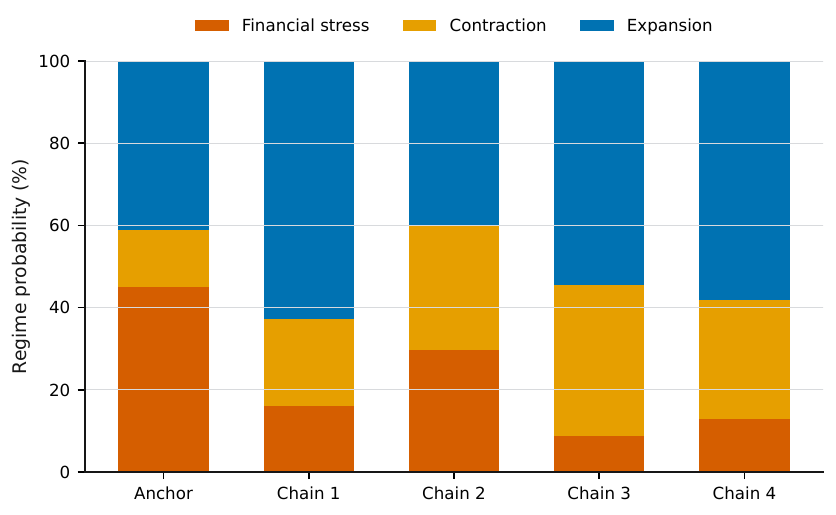}
\caption{Fixed PAO diversification: regime classification of the anchor and generated states.}
\label{fig:pao_diversify_regime}
\end{figure}

\begin{figure}[h]\centering
\includegraphics[width=0.95\linewidth]{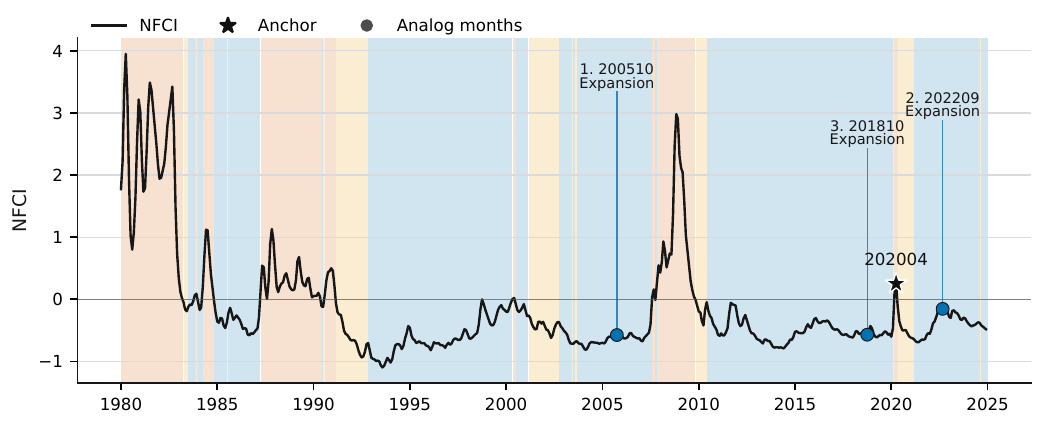}
\caption{Fixed PAO diversification: NFCI history with the nearest historical analogues.}
\label{fig:pao_diversify_nfci}
\end{figure}

\paragraph{Regime-oriented comparison: summer-child versus winter-wolf.}
\label{par:sc_ww_probe}
A portfolio decision pipeline is shaped not only by its architecture and optimization layer, but also by the market environments emphasized during training. We compare two fixed PAO pipelines that share the same construction protocol and robust optimization layer, but differ in their training exposure: a policy trained on comparatively stable periods, \emph{summer-child} (SC), and a stress-oriented policy trained on a history that includes major crises, \emph{winter-wolf} (WW). The names are mnemonic; both are treated as fixed trained pipelines. We ask:
\begin{quote}
\emph{Under which macroeconomic states does the SC pipeline outperform the WW pipeline at a
stress-period anchor date?}
\end{quote}
We use April~2020 as the anchor, an acute COVID-19 stress episode. At the observed anchor state, the stress-oriented WW pipeline attains the higher realized excess return, $r_{\mathrm{SC}}=8.05\%$ against $r_{\mathrm{WW}}=9.37\%$, an SC$-$WW gap of $-1.32$ percentage points. With $r_{j}(\vec{m})=\vec{w}_{j}(\vec{m})^{\top}\vec{r}_{\bar t+1}$ for $j\in\{\mathrm{SC},\mathrm{WW}\}$, the probe uses the decision-level objective
\begin{equation}
\label{eq:sc_ww_probe}
G_{\mathrm{SC>WW}}(\vec{m}) = -\bigl(r_{\mathrm{SC}}(\vec{m})-r_{\mathrm{WW}}(\vec{m})\bigr),
\end{equation}
\noindent so that low values of $G$ correspond to states under which SC outperforms WW. Economic plausibility is supplied by the prior $p_{0}$ localized around the April~2020 anchor. 

The ranking reverses in all four chains. At the generated states, SC outperforms WW by $7.9$, $8.1$, $7.9$, and $7.7$ percentage points with a median gap of $+7.9$ percentage points. Thus, the stress-oriented WW advantage observed at the April~2020 anchor does not persist uniformly across nearby plausible macroeconomic states.

Figure~\ref{fig:sc_ww_density} summarizes the macroeconomic states under which SC overtakes WW. The generated states are not uniformly calmer than the April~2020 anchor, but they differ from the COVID-19 shock in economically meaningful ways. Stock-market variance $svar$ falls relative to the anchor, indicating lower market volatility, while book-to-market ratio $bm$ and dividend--price ratio $dp$ rise, indicating cheaper equity valuations. Inflation $infl$ and the Treasury bill rate $tbl$ are also higher. Taken together, these movements describe a weaker but less panic-like macro-financial environment: closer to a slower, rate-driven downturn than to the sudden high-volatility COVID-19 shock.

Figures~\ref{fig:sc_ww_regime} and~\ref{fig:sc_ww_nfci} provide additional interpretation. The regime classifier in Figure~\ref{fig:sc_ww_regime} classifies the generated states predominantly as contraction, whereas the April~2020 anchor reflects a more mixed stress environment. The nearest historical analogues in Figure~\ref{fig:sc_ww_nfci} include downturn-like episodes associated with interest-rate pressure rather than acute market panic, including the late-2018 selloff during the Federal Reserve's rate-hiking cycle, the end-2015 first-hike period, and a 2020 contraction month.

The main implication is that the value of stress-oriented training depends on the type of adverse macroeconomic environment. WW performs better at the April~2020 anchor, which corresponds to a sudden, high-volatility crisis episode. However, under generated states resembling a slower, rate-driven contraction, SC overtakes WW by about eight percentage points across the four chains. The result should not be read as showing that either policy is uniformly safer or superior. Rather, it shows that the relative ranking of the two trained pipelines is state-dependent: different forms of macroeconomic stress can favor different learned decision rules.
\begin{figure}[h]\centering
\includegraphics[width=0.95\linewidth]{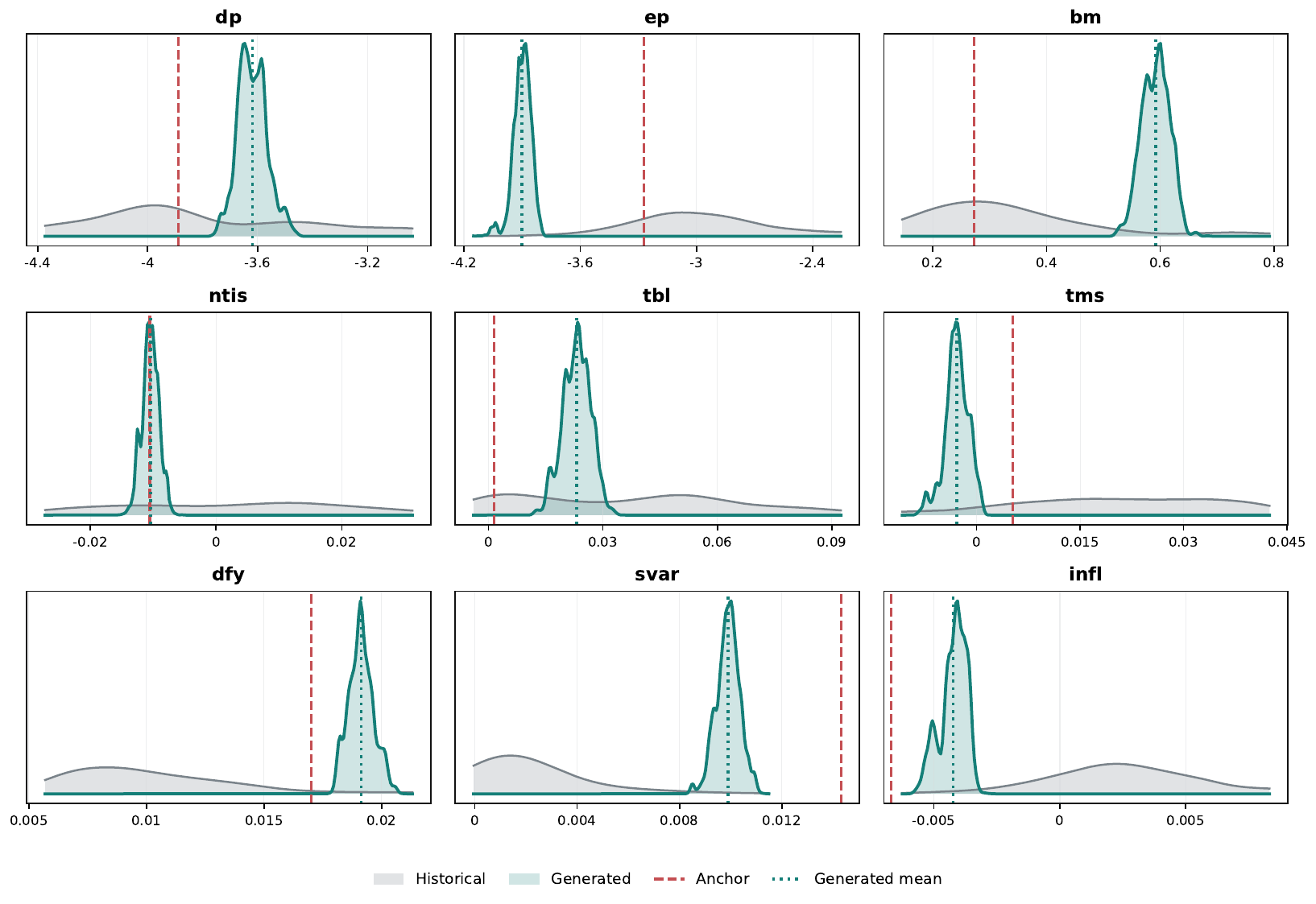}
\caption{SC versus WW: generated macroeconomic density relative to the historical
panel.}
\label{fig:sc_ww_density}
\end{figure}

\begin{figure}[h]\centering
\includegraphics[width=0.78\linewidth]{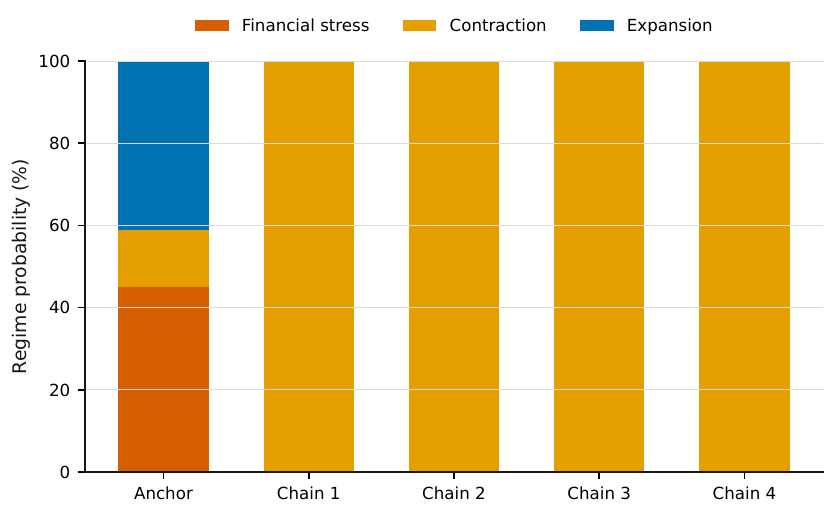}
\caption{SC versus WW: regime classification of the anchor and generated states.}
\label{fig:sc_ww_regime}
\end{figure}

\begin{figure}[h]\centering
\includegraphics[width=0.95\linewidth]{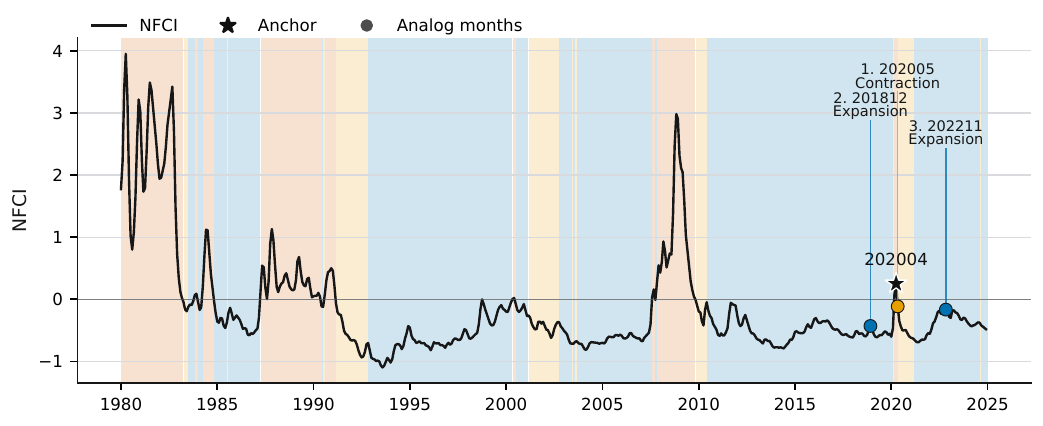}
\caption{SC versus WW: NFCI context and nearest historical analogues.}
\label{fig:sc_ww_nfci}
\end{figure}

\paragraph{Targeting a benchmark return.}
\label{par:pto_benchmark_return}
The preceding scenarios contrasted two pipelines or probed allocation structure; we now invert a single PTO pipeline against an explicit performance target, the most directly interpretable use of the framework. Given a benchmark excess return $b$, the probe asks which macroeconomic states would have made the pipeline realize that return. Rather than pose this once, we pose it at two anchors of opposite character, a calm, late-cycle month and a banking-stress month, and read the two answers together. We ask:

\begin{quote}
\emph{Under which macroeconomic states would the PTO pipeline have realized a portfolio return close to a benchmark $b$?}
\end{quote}
With $r_{\mathrm{PTO}}(\vec{m})=\vec{w}_{\mathrm{PTO}}(\vec{m})^{\top}\vec{r}_{\bar t+1}$, the probe uses the benchmark-return objective of Section~\ref{subsec:probingfunctions},
\eqref{eq:benchmarkG},
\begin{equation*}
G(\vec{m}) = \bigl(r_{\mathrm{PTO}}(\vec{m}) - b\bigr)^{2},
\end{equation*}
which is minimized when the realized PTO return matches $b$; plausibility is supplied by the prior $p_{0}$ localized around each anchor. Read across the two anchors, the probe does more than translate a target into a scenario: it traces the pipeline's how far the realized return can be pushed before no economically plausible state will carry it further.

\emph{A benign anchor (October~2019): both benchmarks are within reach.} October~2019 is a calm, late-cycle expansion month in which the PTO pipeline realizes an excess return of $-0.68\%$, having returned $+1.59\%$ in the preceding month (September~2019) while the market (S\&P~500 excess) returned
$+0.36\%$. The pipeline thus lagged both the market and its own recent record in an otherwise quiet environment. We target the two benchmarks it fell short of, and both are caught. Across the four chains the realized return reaches the market benchmark at $0.35$, $+0.40$, $+0.28$, and $+0.31\%$, and the repeat last month benchmark at $+1.56$, $+1.45$, $+1.50$, and $+1.63\%$.

Figure~\ref{fig:pto_benchmark_201910_density} shows what the macroeconomic inputs do to lift the return, and the picture is an economically familiar one: the variables long used to forecast the equity premium all move in the return-raising direction \citep{GoyalWelch2008}. Stocks become cheaper as well as the dividend--price, earnings--price, and book-to-market ratios ($dp$, $ep$, $bm$) all  rise, and risk premia widen, with the default spread $dfy$ and the term spread $tms$ both increasing. The shifts are mild, and the regime model (Figure~\ref{fig:pto_benchmark_201910_regime}) keeps all four chains in the expansion state, while the closest months in history (Figure~\ref{fig:pto_benchmark_201910_nfci}) are calm mid-2000s expansion dates (2005--2007), sitting at the same low level of financial stress as the anchor itself. In plain terms, only a modest, plausible repricing of risk is needed for the pipeline to catch up,and the probe, given nothing but a return target, independently rediscovers the classical drivers of high expected returns.

\emph{A stressed anchor (March~2023): the market benchmark is a ceiling.} March~2023 is a banking-stress month in which the pipeline realizes $-1.28\%$, having returned $+1.93\%$ in February~2023 while the S\&P~500 returned $+3.28\%$ in excess of the risk-free rate. Here the two benchmarks come apart. The ``repeat last month'' target is largely recovered,the chains reach $+0.94$, $+1.37$, $+1.19$, and $+1.31\%$ (median $+1.25\%$, within $0.68$~pp of $+1.93\%$) under states that stay empirically plausible. The market benchmark, however, is not reached at all: every chain plateaus at almost the same place,$+1.41$, $+1.59$, $+1.56$, and $+1.63\%$ (median $+1.57\%$), leaving a residual gap of $1.71$~pp that does not close. That the four independent chains agree so tightly on the same shortfall is the signature of a genuine ceiling rather than sampler  noise. Figure~\ref{fig:pto_benchmark_density} shows how hard the sampler tries: net equity issuance $ntis$ and the term spread $tms$ collapse by more than three standardized units, the short rate $tbl$ and the default spread $dfy$ climb, and the representative states abandon the anchor regime altogether, turning into financial stress in all four chains (Figure~\ref{fig:pto_benchmark_regime}). The generated state has no close historical analogue at all: the nearest distinct months (Figure~\ref{fig:pto_benchmark_nfci}) are scattered years away, early~2007, early~2019, and early~2024, and even these lie about three times farther in standardized macro space than the October~2019 analogues. The optimizer pushes the macroeconomic state to the very edge of what the historical record will support and still cannot manufacture the market return.

Taken together, the two anchors turn a performance target into a feasibility diagnosis. At the benign October~2019 anchor the attainable-return is wide: both the market and the pipeline's recent record sit inside it, reachable through mild, plausible shifts that recover the canonical return predictors. At the stressed March~2023 anchor it collapses: only the modest own-record target is approachable, while the market return lies outside it, unreachable even when the macro state   is driven to the limit of plausibility. Crucially, the probe varies only the macroeconomic inputs while holding the firm characteristics and the trained pipeline fixed; the March~2023 shortfall is therefore not an artifact of an unlucky target but a property of the pipeline itself, given the firms it faced and the weights it had learned, \emph{no} plausible macroeconomic state would have let it match the market. 
Therefore using the framework we were able to recognize that the pipelines fail to achieve what S\&P500 index has realized regardless of the macroeconomic variables. 


\begin{figure}[htbp]\centering
\includegraphics[width=0.95\linewidth]{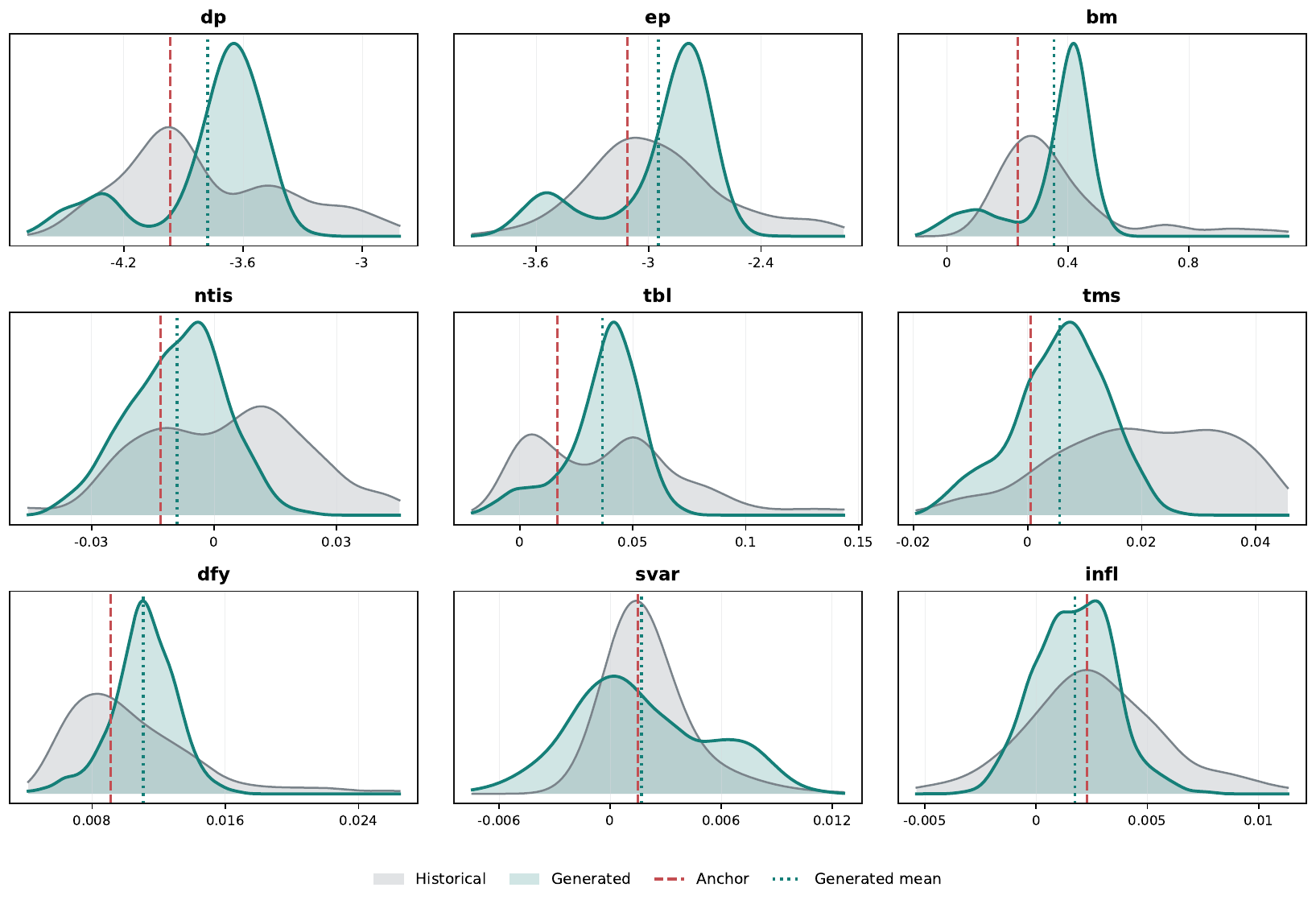}
\caption{Benchmark return, October~2019 (catch-up): generated macroeconomic density against the
historical panel, by variable.}
\label{fig:pto_benchmark_201910_density}
\end{figure}

\begin{figure}[htbp]\centering
\includegraphics[width=0.78\linewidth]{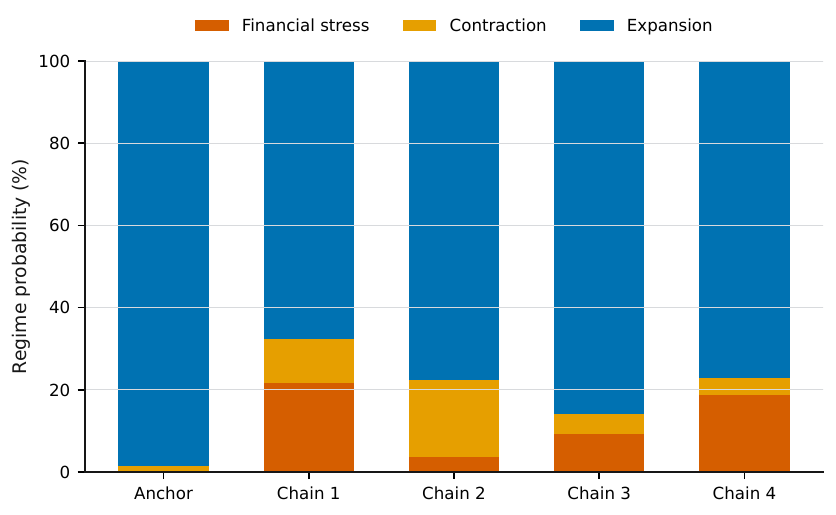}
\caption{Benchmark return, October~2019 (catch-up): regime-probability composition of the anchor and
of each generated chain.}
\label{fig:pto_benchmark_201910_regime}
\end{figure}

\begin{figure}[htbp]\centering
\includegraphics[width=0.95\linewidth]{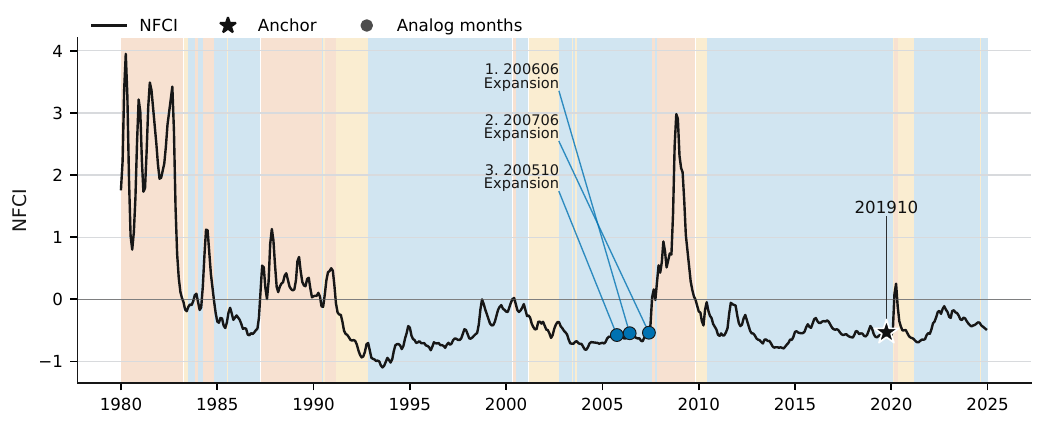}
\caption{Benchmark return, October~2019 (catch-up): NFCI history with the nearest generated-state
analog months.}
\label{fig:pto_benchmark_201910_nfci}
\end{figure}

\begin{figure}[htbp]\centering
\includegraphics[width=0.95\linewidth]{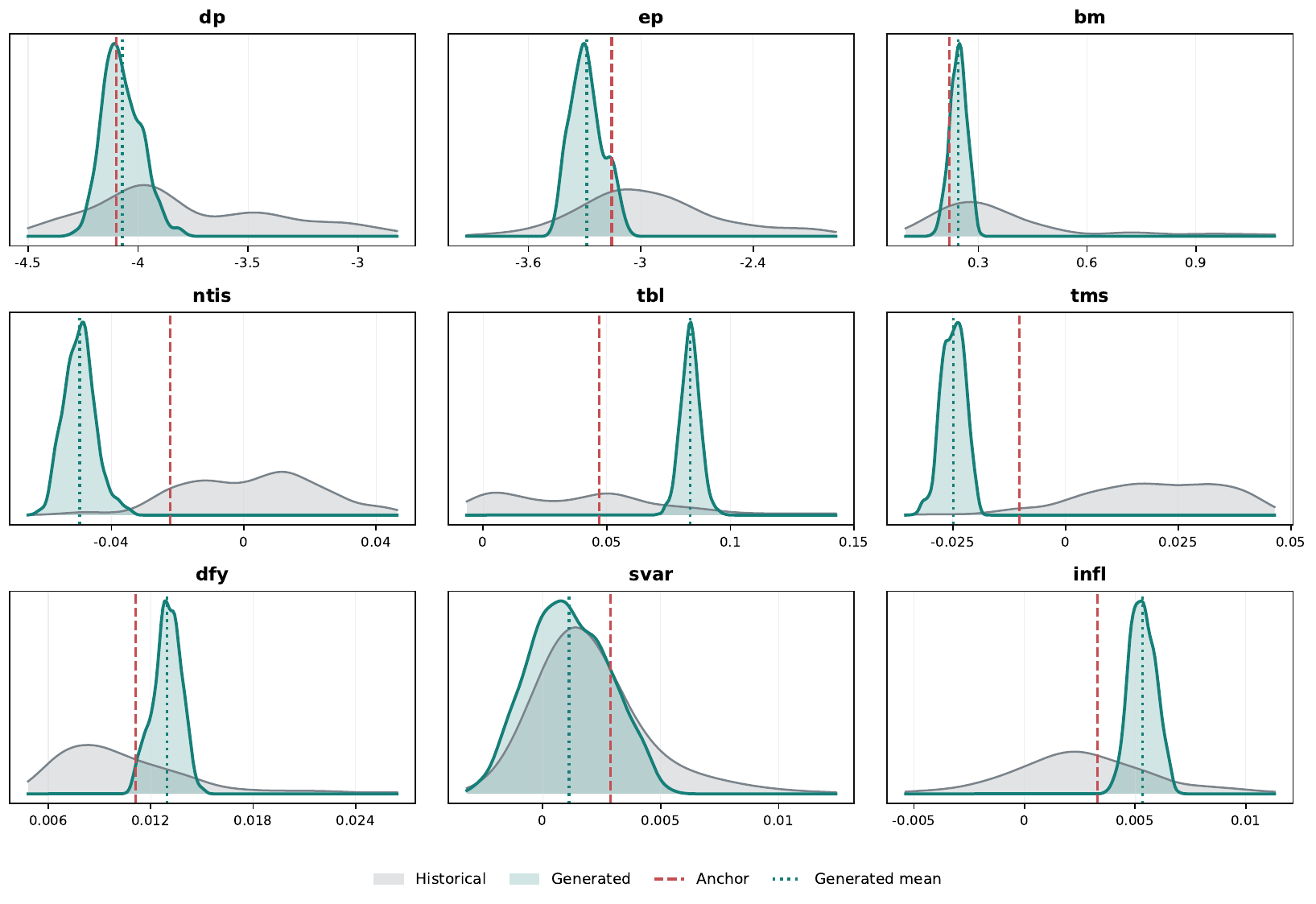}
\caption{Benchmark return, March~2023: generated macroeconomic density against the
historical panel, by variable (S\&P benchmark).}
\label{fig:pto_benchmark_density}
\end{figure}

\begin{figure}[htbp]\centering
\includegraphics[width=0.78\linewidth]{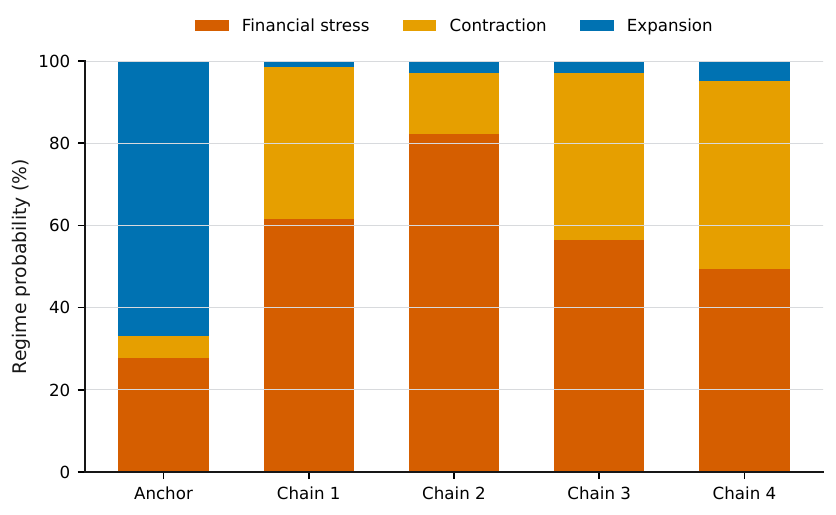}
\caption{Benchmark return, March~2023: regime-probability composition of the anchor and of
each generated chain.}
\label{fig:pto_benchmark_regime}
\end{figure}

\begin{figure}[htbp]\centering
\includegraphics[width=0.95\linewidth]{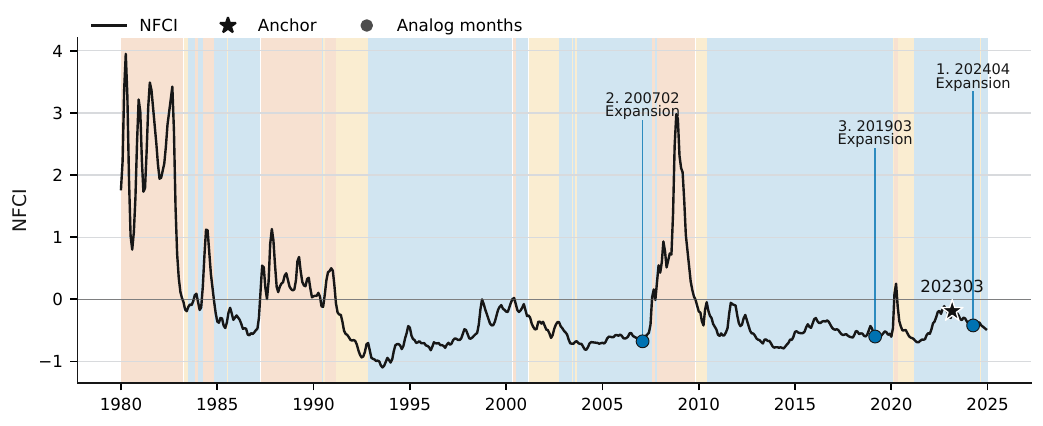}
\caption{Benchmark return, March~2023: NFCI history with the nearest generated-state analog
months.}
\label{fig:pto_benchmark_nfci}
\end{figure}

\section{Conclusions.}
\label{sec:conclusions}

In this work, we introduce a novel \emph{predict-optimize-explain} framework for interpreting portfolio optimization pipelines using targeted scenario generation, yielding fresh insights into how machine learning-driven portfolios behave under different conditions. By formulating what-if questions at the decision level, we gain economically meaningful insights into its behavior, going beyond local feature attribution methods by revealing how decision related quantities require global changes to conditions. This allows us to get functional insights on the entire allocation pipeline, improving the user's understanding of how all the pieces of the pipeline fit together by generating data points for which the user has economic intuition and historical knowledge. 
For example we found that when models are probed to increase diversification, they identify scenarios (e.g., a flatter yield curve with low default return spreads) that compel broader asset allocation, and notably, such diversification tended to improve portfolio returns. This observation reinforces classic portfolio principles on the benefits of diversification, now seen through the lens of an AI-driven decision system. Furthermore, our contrastive experiment showed that a model trained on a full history including major crises (the ``winter-wolf'') generally outperforms a model trained only on stable periods (the ``summer-child'') when a new crisis strikes. This highlights the importance of designing models with awareness of different macroeconomic regimes: incorporating diverse market scenarios in training can yield more robust and reliable portfolio decisions under stress.

Our predict-optimize-explain approach provides a practical tool for both researchers and practitioners to interrogate and interpret various portfolio models. By revealing which economic conditions significantly influence a model’s decisions, the framework enhances transparency and trust in model outputs. For portfolio managers and decision-makers, these insights offer guidance on model selection and deployment, suggesting that integrating decision-focused learning and ensuring exposure to varied market regimes can lead to more resilient investment strategies. Overall, bridging predictive analytics with optimization and explainability can improve not only the performance of portfolio models but also their interpretability and robustness, paving the way for more informed decision-making in asset management.

Future work may explore several natural extensions of the proposed framework. While our experiments focused on probing macroeconomic variables to retain interpretability, the methodology extends seamlessly to other input dimensions. For instance, one could fix macroeconomic conditions and generate synthetic firm characteristic vectors to investigate how the introduction of new or altered firms would influence portfolio allocations. This would enable a richer understanding of how models respond to changes in the asset universe. In addition, instead of regularizing around a specific macroeconomic snapshot, future implementations might incorporate trajectory-level regularizers to reflect consistency with historical evolution over time. Lastly, although our study emphasizes decision pipelines with explicit optimization layers, the framework is equally applicable to end-to-end, model-free strategies that directly predict allocation weights, as in recent advances. These directions offer promising opportunities to further align explainable decision-making with realistic market settings.

\clearpage

\bibliographystyle{abbrvnat}
\bibliography{POE}

\clearpage

\appendix

\section{Macroeconomic Predictors.}
\label{app:macro_preds}

We reserve this appendix to give the definitions of the macroeconomic predictors that we have used throughout our computational study (see Table \ref{tab:macro_defs}).

\begin{table}[H]
\centering
\small
\caption{Macroeconomic predictors from the Goyal--Welch dataset and their definitions.}
\label{tab:macro_defs}
\begin{tabular}{@{}l l p{0.72\linewidth}@{}}
\toprule
Symbol & Name & Definition / construction \\
\midrule
$dp$ & Dividend--price ratio &
$\log(D_t)-\log(P_t)$, where $D_t$ is the trailing 12-month sum of dividends of the S\&P500 index and $P_t$ is its price at time $t$. \\

$ep$ & Earnings--price ratio &
$\log(E_t)-\log(P_t)$, where $E_t$ is the trailing 12-month sum of earnings of the S\&P500 index and $P_t$ is the index price level. \\

$bm$ & Book-to-market ratio &
Book value divided by market value for the Dow Jones Industrial Average (with standard reporting-lag convention for book value). \\

$ntis$ & Net equity expansion &
Net equity issuance activity scaled by aggregate market capitalization; ratio of 12-month moving sum of net equity issues by New York Stock Exchange (NYSE) divided by total end-of-year market capitalization of NYSE stocks. \\

$tbl$ & Treasury bill rate &
Short-term Treasury bill rate canonical short-rate proxy; interest rate on the the 3-month T-bill. \\

$tms$ & Term spread &
$tms_t = lty_t - tbl_t$, i.e., long-term government bond yield minus the Treasury bill rate. \\

$dfy$ & Default spread &
$dfy_t = y^{\text{BAA}}_t - y^{\text{AAA}}_t$, i.e., the yield difference between Moody's BAA and AAA-rated corporate bonds. \\

$svar$ & Stock market variance &
Realized market variance from daily index returns within month $t$, typically computed as the within-month sum of squared daily returns. \\

$infl$ & Inflation rate &
Inflation computed from the Consumer Price Index (CPI). \\
\bottomrule
\end{tabular}

\vspace{2mm}
\footnotesize \noindent
\textit{Notes:} Definitions follow the constructions described in \citet{GoyalWelch2008}. The time series are obtained from the updated Goyal--Welch monthly predictor dataset \citep{GoyalWelchData2024}.
\end{table}

The macroeconomic predictors of \citet{GoyalWelch2008} have well-established economic interpretations, which we use when reading the generated scenarios. The valuation ratios $dp$, $ep$, and $bm$ (dividend-, earnings-, and book-to-market) are countercyclical and positively related to expected returns: high values typically signal that the market is priced cheaply relative to fundamentals, often in or after downturns. Net equity issuance $ntis$ is procyclical and negatively related to subsequent returns, as firms tend to issue equity when valuations are high. Among the interest-rate variables, the term spread $tms$ (long-term yield minus the Treasury-bill rate) is a classic leading indicator of the business cycle: a flat or inverted curve has historically preceded recessions \citep{estrella1996yield}. The default spread $dfy$ (the BAA$-$AAA corporate yield difference) and the stock-market variance $svar$ both rise in periods of financial stress and elevated risk premia; widening credit spreads, in particular, are a forward-looking measure of deteriorating financial conditions \citep{GilchristZakrajsek2012}. The Treasury-bill rate $tbl$ summarizes the monetary-policy stance, and inflation $infl$ captures the price environment. Read jointly, a configuration of falling $tbl$, rising $tms$, widening $dfy$, and elevated $svar$ corresponds to a contraction-like, credit-stressed state in which the central bank eases as financial conditions tighten, the regime that recurs in our predict-then-optimize versus predict-and-optimize comparison.

\section{Details of Training Process.}
\label{app:training_details}

\noindent This appendix collects implementation details that support replication but are secondary
to the main narrative in Section~\ref{sec:NumEx}. Table~\ref{tab:setup_params} summarizes the experimental
setup (data split, universe construction, EWMA risk estimation, and training hyperparameters).
We also report the supplementary out-of-sample performance tables and full parameter sweeps referenced
in Section~\ref{sec:NumEx}.

\begin{table}[h]
\centering
\caption{Experimental setup parameters.}
\label{tab:experimental_setup}
\small
\begin{tabular}{lll}
\toprule
\textbf{Category} & \textbf{Parameter} & \textbf{Setting} \\
\midrule
\multirow{4}{*}{Data and universe}
& Universe & curated 500 firms (S\&P~100 + large-cap diversified) \\
& Return lookback ($L$) & 60 months \\
& Rebalancing & monthly \\
\midrule
\multirow{3}{*}{Covariance estimation}
& EWMA decay ($\beta$) & 0.94 \\
& Diagonal shrinkage ($\delta$) & 0.10 \\
& Ridge term ($\epsilon$) & $10^{-6}$ \\
\midrule
\multirow{3}{*}{Neural network}
& Architecture & $1400 \rightarrow 32 \rightarrow 16 \rightarrow 8 \rightarrow 1$ \\
& Nonlinearity & ReLU + BatchNorm \\
& Dropout & 0.5 \\
& Patience & 10 epochs \\
\midrule
\multirow{4}{*}{Training}
& Optimizer & AdamW \\
& Learning rate & $5\times 10^{-5}$ \\
& Weight decay & $10^{-5}$ \\
& Max epochs / patience & 50 / 10 \\
\bottomrule
\label{tab:setup_params}
\end{tabular}
\end{table}

We first present full parameter sweeps for the PTO and PAO training strategies across robustness configurations $(\kappa,\boldsymbol{\Omega})$ and risk-aversion levels $\lambda$, under the common evaluation protocol described in Section~\ref{subsec:design}. Results are summarized by annualized mean excess return (Mean$_a$), annualized volatility (Vol$_a$), annualized Sharpe ratio (Sharpe$_a$), maximum drawdown (MaxDD), and terminal wealth (Wealth$_T$).

Table~\ref{tab:pto_excess_summary} reports OOS performance for the PTO strategy, where portfolio weights are obtained by solving the robust mean--variance problem~\eqref{eqn:genopt}. The main pattern is that a \emph{moderate} mean-uncertainty penalty improves risk-adjusted performance relative to both the MVO special case ($\kappa=0$) and the equal-weight benchmark. In particular, under $\boldsymbol{\Omega}=\mathrm{diag}(\boldsymbol{\Sigma})$, the configuration $\kappa=0.1$ attains the strongest Sharpe$_a$ (1.1965), while simultaneously reducing realized volatility and improving drawdown relative to the benchmark (MaxDD improves from $-0.3460$ to $-0.2603$). Increasing $\kappa$ beyond this range attenuates responsiveness to the return signal and correspondingly reduces Sharpe$_a$ (e.g., Sharpe$_a$ decreases from 1.1965 at $\kappa=0.1$ to 1.0812 at $\kappa=1.0$ under $\boldsymbol{\Omega}=\mathrm{diag}(\boldsymbol{\Sigma})$). The choice of uncertainty scaling is also consequential. Across $\kappa>0$, $\boldsymbol{\Omega}=\mathrm{diag}(\boldsymbol{\Sigma})$ dominates $\boldsymbol{\Omega}=\mathbf{I}$ in Sharpe$_a$ and drawdown in this experiment, consistent with the interpretation that diagonal scaling penalizes mean-uncertainty in proportion to asset-specific risk. Finally, as expected from model~\eqref{eqn:genopt}, when $\kappa=0$ the uncertainty term vanishes and performance is effectively independent of $\boldsymbol{\Omega}$ (the MVO special case).

\begin{table}[htbp]
\centering
\caption{PTO out-of-sample performance (monthly excess returns), January 2016--November 2024.}
\label{tab:pto_excess_summary}
\small
\begin{tabular}{lccccc}
\toprule
Decision Pipeline & Mean$_a$ & Vol$_a$ & Sharpe$_a$ & MaxDD & Wealth$_T$ \\
\midrule
EqualWeight & 0.1749 & 0.2185 & 0.8006 & $-$0.3460 & 3.8221 \\
$\boldsymbol{\Omega}=\mathrm{diag}(\boldsymbol{\Sigma})$, $\kappa=0.0$  & 0.1691 & 0.1749 & 0.9667 & $-$0.3228 & 3.9050 \\
$\boldsymbol{\Omega}=\mathrm{diag}(\boldsymbol{\Sigma})$, $\kappa=0.1$  & \textbf{0.1835} & \textbf{0.1533} & \textbf{1.1965} & \textbf{$-$0.2603} & \textbf{4.5724} \\
$\boldsymbol{\Omega}=\mathrm{diag}(\boldsymbol{\Sigma})$, $\kappa=0.5$  & 0.1763 & 0.1523 & 1.1576 & $-$0.2852 & 4.2896 \\
$\boldsymbol{\Omega}=\mathrm{diag}(\boldsymbol{\Sigma})$, $\kappa=1.0$  & 0.1737 & 0.1606 & 1.0812 & $-$0.2981 & 4.1432 \\
$\boldsymbol{\Omega}=\mathrm{diag}(\boldsymbol{\Sigma})$, $\kappa=10.0$ & 0.1798 & 0.1798 & 0.9999 & $-$0.3132 & 4.2519 \\
$\boldsymbol{\Omega}=\mathbf{I}$, $\kappa=0.0$   & 0.1691 & 0.1749 & 0.9667 & $-$0.3228 & 3.9050 \\
$\boldsymbol{\Omega}=\mathbf{I}$, $\kappa=0.1$   & 0.1797 & 0.1823 & 0.9856 & $-$0.3074 & 4.2299 \\
$\boldsymbol{\Omega}=\mathbf{I}$, $\kappa=0.5$   & 0.1754 & 0.2027 & 0.8656 & $-$0.3234 & 3.9445 \\
$\boldsymbol{\Omega}=\mathbf{I}$, $\kappa=1.0$   & 0.1754 & 0.2100 & 0.8353 & $-$0.3339 & 3.8942 \\
$\boldsymbol{\Omega}=\mathbf{I}$, $\kappa=10.0$  & 0.1750 & 0.2176 & 0.8043 & $-$0.3447 & 3.8302 \\
\bottomrule
\end{tabular}
\end{table}

Table~\ref{tab:pao_top10} reports the top ten PAO strategies ranked by test Sharpe$_a$. Relative to PTO, the strongest PAO configurations can improve Sharpe but exhibit objective-dependent mean/volatility and drawdown trade-offs. 
Two regularities stand out. First, the best-performing PAO pipelines (by Sharpe$_a$) are obtained under the utility training objective with $\boldsymbol{\Omega}=\mathrm{diag}(\boldsymbol{\Sigma})$ and $\kappa=1.0$ (Sharpe$_a$ of 1.291 and 1.242 for $\lambda=5$ and $\lambda=10$, respectively). Second, within the top set, robustness is not incidental: $\kappa=1.0$ appears in the majority of the top-10 entries, suggesting that explicit mean-uncertainty penalization plays a central stabilizing role. At the same time, objective choice materially affects tail behavior. For example, the strongest return-trained configuration in Table~\ref{tab:pao_top10} achieves a high Sharpe$_a$ (1.177) but exhibits the largest MaxDD among the top set ($-0.401$), illustrating that maximizing average return can induce more adverse drawdown profiles even when average risk-adjusted performance remains competitive.

\begin{table}[htbp]
\centering
\caption{Top 10 PAO decision pipelines ranked by annualized Sharpe, , January 2016–November 2024}
\label{tab:pao_top10}
\small
\begin{tabular}{llrrcccc}
\toprule
Loss & $\boldsymbol{\Omega}$ & $\lambda$ & $\kappa$ & Sharpe$_a$ & Mean$_a$ & Vol$_a$ & MaxDD \\
\midrule
utility & $\mathrm{diag}(\boldsymbol{\Sigma})$ &  5  & 1.0 & 1.291 & 0.282 & 0.219 & $-$0.295 \\
utility & $\mathrm{diag}(\boldsymbol{\Sigma})$ & 10  & 1.0 & 1.242 & 0.329 & 0.265 & $-$0.273 \\
return  & $\mathrm{diag}(\boldsymbol{\Sigma})$ &  5  & 1.0 & 1.177 & 0.254 & 0.216 & $-$0.401 \\
sharpe  & $\mathbf{I}$                      & 20  & 1.0 & 1.168 & 0.204 & 0.175 & $-$0.295 \\
utility & $\mathbf{I}$                      & 20  & 1.0 & 1.106 & 0.197 & 0.178 & $-$0.311 \\
utility & $\mathbf{I}$                      & 10  & 1.0 & 1.091 & 0.209 & 0.191 & $-$0.309 \\
sharpe  & $\mathbf{I}$                      & 10  & 1.0 & 1.084 & 0.198 & 0.183 & $-$0.315 \\
utility & $\mathbf{I}$                      & 20  & 0.1 & 1.081 & 0.245 & 0.227 & $-$0.330 \\
sharpe  & $\mathbf{I}$                      & 20  & 0.1 & 1.081 & 0.242 & 0.224 & $-$0.277 \\
sharpe  & $\mathbf{I}$                      &  5  & 1.0 & 1.025 & 0.191 & 0.186 & $-$0.326 \\
\bottomrule
\end{tabular}
\end{table}

\begin{figure}[H]
  \centering
  \includegraphics[width=\linewidth]{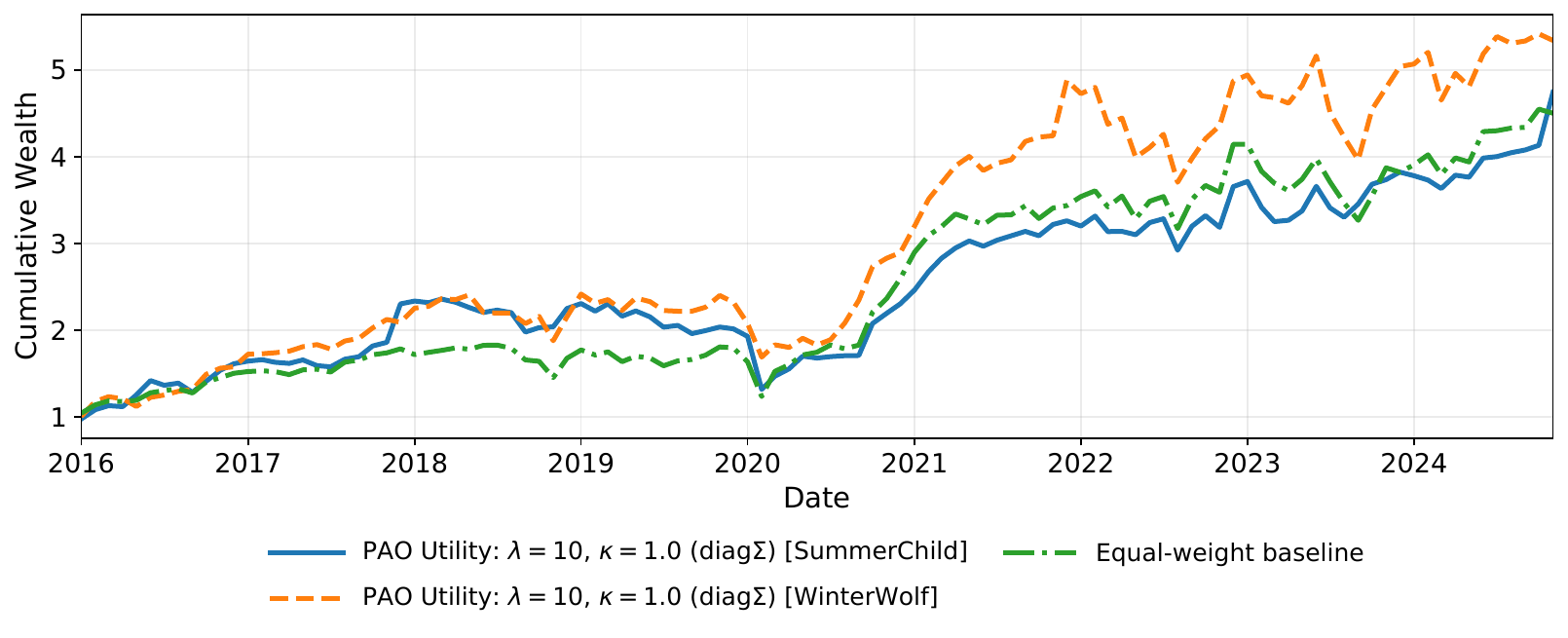}
  \caption{Cumulative wealth over the test period for the equal-weight benchmark and two PAO pipelines trained with the utility objective: the ``summer-child'' and ``winter-wolf'' models.}
  \label{fig:app_wealth_summer_winter}
\end{figure}

\begin{figure}[H]
  \centering
  \includegraphics[width=\linewidth]{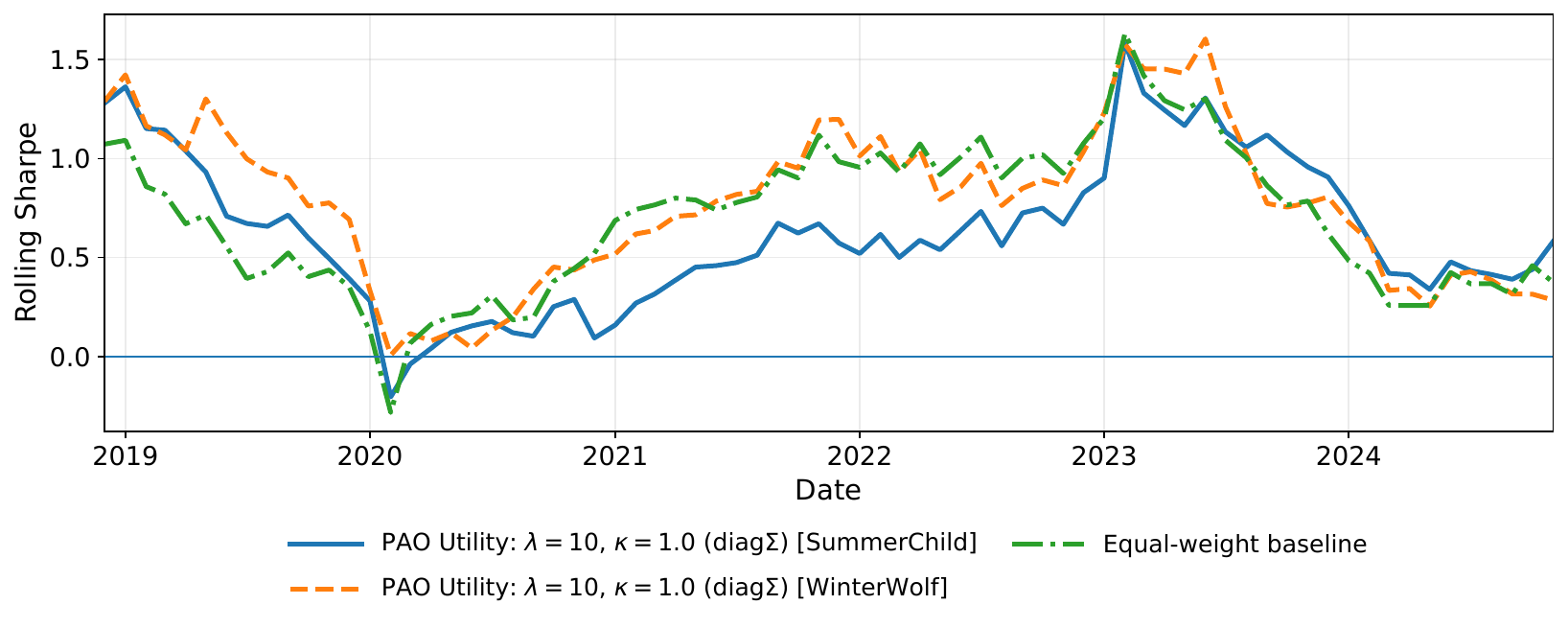}
  \caption{Rolling Sharpe ratios (annualized; 36-month window) computed for the equal-weight benchmark and two PAO utility-trained pipelines.}
  \label{fig:app_sharpe_summer_winter}
\end{figure}

\clearpage

\section{Economic Interpretation: Diagnostic Tools}
\label{app:econ_interp}

To read the generated macroeconomic scenarios in economic terms, we use three
auxiliary tools. We emphasize that these are \emph{post-hoc interpretive
diagnostics}: none of them enters the probing function $G$ or the plausibility
prior $p_0$, so they do not affect which scenarios are generated, only how we
describe them.

\paragraph{Regime classifier.}
We label each historical month as \emph{expansion}, \emph{contraction}, or
\emph{financial stress} using a rule grounded in standard business-cycle and
financial-conditions indicators: a month is labeled financial stress when the
Chicago Fed National Financial Conditions Index is above its historical average;
contraction when an NBER recession is in progress, the Sahm-rule unemployment
trigger fires, or activity is weak with labor-market slack; and expansion
otherwise. We then fit a multinomial logistic classifier from the nine
standardized macroeconomic predictors to these labels; it attains an out-of-sample
classification accuracy of about $95\%$. When we report the regime of a generated
scenario, we apply this classifier to the generated state.

\paragraph{Nearest historical analogues.}
To locate a generated scenario relative to history, we report its nearest
historical month(s) under a Mahalanobis distance in standardized macroeconomic
space. This supports statements such as ``the generated region most closely
resembles [month/year].''

\clearpage

\section{Convergence of the MALA Algorithm.}
\label{app:convergenceMALA}

 We refer the reader to \citep{roberts2004general} for an accessible survey of MCMC algorithms and their convergence guarantees. MALA algorithm converges (in law) to the target distribution under standard conditions: irreducibility, aperiodicity, and invariance of the target density. In particular, if the Markov chain satisfies the detailed balance condition with respect to the target distribution $p$ in \eqref{eq:GibbsforG}, then this distribution is stationary for the chain, and (under irreducibility and aperiodicity) it is the unique limiting distribution.

The Metropolis Hastings acceptance-rejection probability $\alpha$ in MALA is designed precisely to ensure that MALA updates satisfy the stationarity condition for the distribution $p$; see Propositions 1 and 2 in \citep{roberts2004general}. The idea behind detailed balance is intuitive: if $T(\vec{m}\to \vec{m}')$ denotes the one-step transition probability of the MALA algorithm, then it must satisfy
\begin{equation} \label{eq:detailedBalance}
    T(\vec{m} \to \vec{m}') p(\vec{m}) = T(\vec{m}' \to \vec{m}) p(\vec{m'}).
\end{equation}
Intuitively, this condition balances probability flow between any pair of states, so that the target density is preserved by the dynamics. In MALA, the transition kernel can be written as $T(\vec{m} \to \vec{m'}) = \Gamma(\vec{m} \to \vec{m'}) \alpha$, where $\Gamma$ is the proposal density and $\alpha = \min\{1, \frac{p(\vec{m}') \Gamma(\vec{m}'\to \vec{m})}{p(\vec{m})\Gamma(\vec{m} \to \vec{m}')}\}$. Let the reciprocal acceptance probability be $\alpha' = \min\{1, \frac{p(\vec{m})\Gamma(\vec{m} \to \vec{m}')}{p(\vec{m}') \Gamma(\vec{m}'\to \vec{m})}\}$. Therefore, if the minimum is achieved with the ratio, then $\alpha = 1$, and we have
\begin{align*}
 T(\vec{m} \to& \vec{m}') p(\vec{m}) =  \Gamma(\vec{m} \to \vec{m'}) \alpha p(\vec{m}) \\
 &=  \Gamma(\vec{m} \to \vec{m'})\frac{p(\vec{m}') \Gamma(\vec{m}'\to \vec{m})}{p(\vec{m})\Gamma(\vec{m} \to \vec{m}')} p(\vec{m})\\
 &= p(\vec{m'}) \Gamma(\vec{m}' \to \vec{m}) \\
 &= p(\vec{m}')T(\vec{m}' \to \vec{m}),
\end{align*}
which is exactly \eqref{eq:detailedBalance}. If instead $\alpha = 1$, the argument is symmetric between the right hand side and left hand side. Thus the chain satisfies detailed balance with respect to $p$, implying that $p$ is stationary. Under irreducibility and aperiodicity, this stationary distribution is also the limiting distribution.

The irreducibility condition \cite[p. 31]{roberts2004general} states that, starting from any point $\vec{m}$, the chain can reach a set $A \subset \mathcal{M}$ with positive measure in finitely many steps with positive probability. In our setting, this holds because the MALA proposal includes Gaussian noise with full support, so any such set $A$ can be reacged with nonzero probability in one step.

Aperiodicitiy is a disjoint decomposition $\mathcal{M}_1, \ldots, \mathcal{M}_d \subseteq \mathcal{M}$ such that given $\vec{m}\in \vec{M}_i$, the next step of the chain falls inside $\mathcal{M}_{i + 1}$ (with cyclical modulo $d$ indexing if $i = d$) with probability $1.0$ \cite[p. 31]{roberts2004general}. For MALA, aperiodicity holds because proposals are accepted with probability strictly less than one in general, and rejections imply a positive probability of remaining at the current state, preventing periodic behavior.

Given the conditions of aperiodicity and irreducibility, and stationarity of distribution $p$, then for almost all initial points $\vec{m}_0$, the probability of the chain reaching a measurable subset $A$ at $n$ steps (as $n \to \infty$) is $p(A)$.

For explicit convergence rates (e.g., in Wasserstein distance), see \citep{durmus2022geometric}, which relates geometric convergence to regularity conditions on $G$ (such as Lipschitz continuity) and suitable tail behavior (requiring sufficiently light, super-exponentially decaying tails); see \citep{roberts1996exponential}.

\end{document}